\documentclass[12pt, a4paper]{amsart}
\usepackage{amssymb,fullpage}
\usepackage{amsmath,esint}
\usepackage[colorlinks, allcolors=RedViolet]{hyperref}
\usepackage{colortbl}
\usepackage[dvipsnames]{xcolor}

\newtheorem{theorem}[equation]{Theorem}%[section]
\newtheorem{lemma}[equation]{Lemma}

\theoremstyle{definition}

\theoremstyle{remark}
\newtheorem{remark}[equation]{Remark}

\numberwithin{equation}{section}

\newcommand{ \R }{ \mathbb{R} }

\newcommand{\n}{\nabla}
\renewcommand{\epsilon}{\varepsilon}
\renewcommand{\phi}{\varphi}
\renewcommand{\le}{\leqslant}
\renewcommand{\ge}{\geqslant}
\renewcommand{\leq}{\leqslant}
\renewcommand{\geq}{\geqslant}

\begin{document}

\title[nonlocal equations with variable powers]
{Local H\"older regularity for nonlocal equations with variable powers}


\author{Jihoon Ok}
  %  Address of record for the research reported here
\address{Department of Mathematics, Sogang University, Seoul 04107, Republic of Korea}
\email{jihoonok@sogang.ac.kr}

\thanks{The author was supported by the National Research Foundation of Korea funded by the
Korean Government (NRF-2017R1C1B2010328)}

%    General info

%\date{\today}



\begin{abstract}
We prove the local boundedness and the local H\"older continuity of weak solutions to nonlocal equations with variable orders and exponents under sharp assumptions. 
\end{abstract}

\maketitle

%\section*{This is an unnumbered first-level section head}
%This is an example of an unnumbered first-level heading.

%% The correct journal style for \specialsection is all uppercase; a known bug
%% in amsart.cls prevents this, so input must be uppercase until it is fixed.
%\specialsection*{This is a Special Section Head}
%\specialsection*{THIS IS A SPECIAL SECTION HEAD}
%This is an example of a special section head%
%%%%%%%%%%%%%%%%%%%%%%%%%%%%%%%%%%%%%%%%%%%%%%%%%%%%%%%%%%%%%%%%%%%%%%%%
%\footnote{Here is an example of a footnote. Notice that this footnote
%text is running on so that it can stand as an example of how a footnote
%with separate paragraphs should be written.
%\par
%And here is the beginning of the second paragraph.}%
%%%%%%%%%%%%%%%%%%%%%%%%%%%%%%%%%%%%%%%%%%%%%%%%%%%%%%%%%%%%%%%%%%%%%%%%

%%%%%%%%%%%%%%%%%%%%%%%%%%%%%%%%%%%%%%%%%%%%%%%%%%%%%%%%%%%%%%%%%%%%%%%%
%%%%%%%%%%%%%%%%%%%%%%%%%%%%%%%%%%%%%%%%%%%%%%%%%%%%%%%%%%%%%%%%%%%%%%%%
%%%%%%%%%%%%%%%%%%%%%%%%%%%%%%%%%%%%%%%%%%%%%%%%%%%%%%%%%%%%%%%%%%%%%%%%

\section{\bf Introduction}
We study regularity theory for weak solutions to the following integro-differential equations: 
\begin{equation}\label{mainPDE}
\mathcal L_K u(x) = \mathrm{P.V.} \int_{\R^n} |u(x)-u(y)|^{p(x,y)-2} (u(x)-u(y)) K(x,y)\, dy =0  \quad  \text{in } \ \Omega,
\end{equation}
where $\Omega\subset \R^n$ ($n\geq2$) is open and bounded, and $K: \R^n\times \R^n\to \R$ with $K(x,y)=K(y,x)$  is a suitable kernel with variable order $s(\cdot,\cdot)$ and the variable exponent $p(\cdot,\cdot)$. We note that when $p(x,y)\equiv p$, $s(x,y)\equiv s$  and $K(x,y)=\frac{1}{|x-y|^{n+sp}}$, \eqref{mainPDE} is the ($s$-)fractional $p$-Laplace equation $(-\Delta)^s_pu=0$, and moreover if $p=2$, it is the ($s$-)fractional Laplace equation $(-\Delta)^su=0$.   In particular, we mainly prove the local H\"older continuity of weak solutions to \eqref{mainPDE}, under essentially sharp regularity conditions on $s(\cdot,\cdot)$, $p(\cdot,\cdot)$ and $K(\cdot,\cdot)$. We will present our results with the definition of weak solution in Section~\ref{sec1.1} below. 

In the past two decades, there have been tremendous  amount of researches on nonlocal equations of the type \eqref{mainPDE} and relevant variational problems. We refer to for instance  survey papers \cite{DPV1,Pal1} and monographs \cite{MRS1,BV1} for the history,  development and applications of nonlocal problems. Regarding the regularity theory, Caffarelli and Silvestre \cite{CS1} proved Harnack's inequality for the fractional Laplace equation by using an extension argument.  After this pioneering work, regularity theory for nonlocal equations of the fractional Laplacian type has rapidly  developed. We refer to  for instance \cite{BK1,BK2,CCV1,CS2,CS3,Kass1,Kass3,KMS1,RS1,Sil1}  and related references. In particular,     Caffarelli, Chan and Vasseur   applied De Giorgi's approach to nonlocal equations in \cite{CCV1}. 
More generally,  for the fractional $p$-Laplace equation, Di Castro, Kuusi and Palatucci \cite{DKP1,DKP2} have proved Harnack's inequality and the H\"older continuity of weak solutions by using  De Giorgi's approach. Especially, they introduced the so called nonlocal tail (see Section~\ref{sec2}), and obtained regularity estimates involving this. We refer to \cite{CK1,Co1,DZZ1,GaKi1,KKL1,KKP1,KKP2,KKP3,KKP4,KMS2,Lin1,No1,No2} for researches for nonlocal equations of the fractional $p$-Laplacian type. 

 On the other hand, there have also been research activities on regularity theory for nonlocal equations with nonstandard order and exponent. Bass and Kassmann \cite{BK1,BK2} and Silvestre \cite{Sil1} proved the H\"older continuity and Harnack's inequality for bounded solutions to nonlocal linear equations with kernels whose  prototypes are  $|x-y|^{-s(x)2}$ or $\phi(|x-y|)^{-1}$. We also refer to related researches \cite{Bae1,BaeK1,KKLL1}. In addition, recently De Filippis and Palatucci \cite{DP1} considered nonlocal equations of the double phase type, and proved H\"older continuity for bounded solutions.

The model problem we have in mind is the $s(\cdot,\cdot)$-fractional $p(\cdot,\cdot)$-Laplace equation:
\begin{equation}\label{fractionpx}
(-\Delta)^{s(\cdot,\cdot)}_{p(\cdot,\cdot)}u = 0 \quad \text{in }\ \Omega,
\end{equation}
that is, $K(x,y)=\frac{1}{|x-y|^{n+s(x,y)p(x,y)}}$ in \eqref{mainPDE}, where $s(\cdot,\cdot)$ and $p(\cdot,\cdot)$ are variable functions satisfying \eqref{s1s2} and \eqref{gamma}, respectively.  In this paper, we call $s(\cdot,\cdot)$ the variable order, $p(\cdot,\cdot)$  the variable exponent, and both $s(\cdot,\cdot)$ and $p(\cdot,\cdot)$ the variable powers. Note that the nonlocal equation \eqref{fractionpx} is the Euler-Lagrange equations of  the minimizing problem of the energy functional
\[
v\ \ \mapsto\ \ \iint_{\mathcal C_\Omega}\frac{1}{p(x,y)} \frac{|v(x)-v(y)|^{p(x,y)}}{|x-y|^{n+s(x,y)p(x,y)}}\,dydx,
\]
where 
\begin{equation}\label{mcdef} 
\mathcal C_\Omega := (\R^n\times \R^n)\, \setminus \, ((\R^n\setminus \Omega)\times (\R^n\setminus \Omega)),
\end{equation}
and its admissible set  is naturally linked to fractional Sobolev space $W^{s(\cdot,\cdot),p(\cdot,\cdot)}$. In recent years, there have been reported a lot of papers studying nonlocal equations of the type  \eqref{fractionpx} and fractional Sobolev spaces with variable powers, see for instance \cite{Bah1,BahR1,CGA1,HK1,KRV1,LKKL1,LH1,KN1,RAA1,XZY1}. Note that we can also find their applications in some of the preceding references.

The local problem corresponding  to \eqref{fractionpx} is the $p(x)$-Laplace equation
\[
\Delta_{p(\cdot)}u:=\mathrm{div} \left(|Du|^{p(x)-2}Du\right)=0,
\] 
that is a typical model of problems with nonstandard growth, and the one of elliptic equations extensively studied over last two decades. We refer to \cite{AM1,AM2,BO1,BOR1,CM1,Fan1,FZ1} for regularity results of problems modeled by the $p(x)$-Laplace equation. In particular, we know  from \cite{FZ1} that the weak solution to the $p(x)$-Laplace equation is locally bounded if $p(\cdot)$ is continuous, and locally H\"older continuous if $p(\cdot)$ is so called log-H\"older continuous, that is,
\[
\sup_{0<r<\frac{1}{2}}\omega_p(r)\ln(1/r)<\infty, 
\] 
where $\omega_p$ is the modulus of continuity of $p(\cdot)$. Note that the above regularity conditions on $p(\cdot)$ are sharp.

On the other hand, for nonlocal equations with variable powers $s(\cdot,\cdot)$ and $p(\cdot,\cdot)$, despite active research on these problems, we cannot find any result on regularity theory, especially the H\"older continuity or Harnack's inequality. In fact, the nonlocal nature makes difficult to prove regularity estimates for nonlocal equations with variable powers. More precisely, even if we investigate the local regularity of weak solutions, we have to estimate integrals over the whole space $\R^n$, that are called the nonlocal tails. By the way, the differences between the supremum and the infimum of the variable powers may be large, hence we can not approximate the nonlocal tails to integrals with constant powers. This requires much more delicate analysis that is not used in local problems with variable exponent.
 
In this paper, we prove the local boundedness and the H\"older continuity for the weak solutions to nonlocal equations modeled by \eqref{fractionpx} in Theorems \ref{thm1}  and \ref{mainthm}, respectively.  These are the natural nonlocal counter parts of the regularity results for the $p(x)$-Laplace equations mentioned above and, in our best knowledge, the first regularity results for  nonlocal equations with not only the variable exponent $p(\cdot,\cdot)$ but also the variable order $s(\cdot,\cdot)$ depending on two space variables. We emphasize that weak solutions are not assumed to be  bounded  in $\R^n$. Instead, we assume that weak solutions belong to the so called tail space with variable power  $L^\infty(\Omega;L^{p(\cdot,\cdot)-1}_{s(\cdot,\cdot)p(\cdot,\cdot)}(\R^n))$, see Section~\ref{sec2}, and prove the desired regularities. Recently, Chaker and Kim \cite{CK2} proved the local H\"older continuity weak solutions to nonlocal problems with variable exponent and constant order where the condition of variable exponent is stronger than the one in this paper. 

We prove the theorems by using De Giorgi's approach, in particular, in \cite{DKP2}. Hence we first obtain a Caccioppoli type estimate and a logarithmic estimate. Then using the Caccioppoli type estimate and the De Giorgi iteration, we prove the local boundedness of weak solution. Finally, using the above results, we prove an oscillation estimate in Lemma~\ref{lemholder}. Technically, we do not use any embedding property of fractional Sobolev spaces with variable powers even in the proof of the existence of weak solution. Instead, we use well-known embedding properties for the fractional Sobolev spaces with constant powers. It is possible by decreasing the order of differentiability, see Lemma~\ref{lempxpq}.  Moreover, it is unclear that  the multiplication of a weak solution and a cut-off function can be directly used as a test function in the weak form of \eqref{mainPDE}, that is usually the first step of the proofs of Caccioppoli type estimates (see the beginning of Section~\ref{sec4}). This problem occurs in the variable exponent case. Nevertheless, we obtain a desired Caccioppoli estimate by using an approximation argument.

Now, we state the main results in this paper.

\subsection{Main results}\label{sec1.1}
  We start with  defining the variable powers $s(\cdot,\cdot)$ and $p(\cdot,\cdot)$, and the kernel $K(\cdot,\cdot)$. Note that a two variable function $f=f(x,y)$ is said to be \textit{symmetric} if $f(x,y)=f(y,x)$ for all $x$ and $y$. Let $s,p: \R^n \times \R^n \to \R$ be symmetric and satisfy that 
\begin{equation}\label{s1s2}
0< s^{-}:= \inf_{(x,y)\in \R^n \times \R^n} s(x,y) \le \sup_{(x,y)\in \R^n \times \R^n} s(x,y) =:s^{+} < 1,
\end{equation}
and
\begin{equation}\label{gamma}
1< p^{-}:= \inf_{(x,y)\in \R^n \times \R^n} p(x,y) \le \sup_{(x,y)\in \R^n \times \R^n} p(x,y) =:p^{+} < \infty,
\end{equation}
and let $K: \R^n\times \R^n \to \R$ be measurable and symmetric, and satisfy that
\begin{equation}\label{Lambda}
 \frac{\Lambda^{-1}}{ |x-y|^{n+s(x,y)p(x,y)}}  \le K(x,y) \le  \frac{\Lambda }{|x-y|^{n+s(x,y)p(x,y)}}  \quad \text{for a.e. }\ (x,y)\in \R^n\times \R^n,
\end{equation}
for some $ \Lambda>1$.

We next state  the definition used in this paper of weak solution to \eqref{mainPDE} with \eqref{s1s2}--\eqref{Lambda}.  For relevant function spaces $\mathbb W^{s(\cdot,\cdot),p(\cdot,\cdot)}(\Omega)$ and $L^\infty(\Omega;L^{p(\cdot,\cdot)-1}_{s(\cdot,\cdot)p(\cdot,\cdot)}(\R^n))$, we will introduce in the next section. We say that $u\in \mathbb W^{s(\cdot,\cdot),p(\cdot,\cdot)}(\Omega)$ is a weak solution to \eqref{mainPDE} if 
\begin{equation}\label{weakform}
 \iint_{\mathcal C_\Omega} |u(x)-u(y)|^{p(x,y)-2}(u(x)-u(y)) (\phi(x)-\phi(y))K(x,y)\,dy\, dx =0
\end{equation}
for every $\phi\in \mathbb W^{s(\cdot,\cdot),p(\cdot,\cdot)}(\Omega)$ with $\phi=0$ a.e. in $\R^n\setminus\Omega$, where $\mathcal C_\Omega$ is defined in \eqref{mcdef}. In addition, we say $u\in \mathbb W^{s(\cdot,\cdot),p(\cdot,\cdot)}(\Omega)$ is a weak subsolution(resp. supersolution) if \eqref{weakform} with replacing ``$=$" by ``$\le$(resp. $\ge$)" holds for every $\phi\in \mathbb W^{s(\cdot,\cdot),p(\cdot,\cdot)}(\Omega)$ with $\phi\ge0$ a.e. in $\R^n$ and $\phi=0$ a.e. in $\R^n\setminus\Omega$. We will discuss the existence and the uniqueness of the weak solution in Section~\ref{sec3}.

Now we  introduce  regularity assumption on $s(\cdot,\cdot)$ and $p(\cdot,\cdot)$. Set
\[%begin{equation}\label{continuity}
\left\{\begin{aligned}
&\omega_{s}(r):=\sup_{B_r\subset \Omega} \sup_{x_1,x_2,y_1,y_2\in \R^n}|s(x_1,y_1)-s(x_2,y_2)|,\\
&\omega_{p}(r):=\sup_{B_r\subset \Omega} \sup_{x_1,x_2,y_1,y_2\in B_r}|p(x_1,y_1)-p(x_2,y_2)|,
\end{aligned}\right.
 \quad \ r\in(0,1/2).
\]%end{equation}
This means  $\omega_s(\cdot)$, or $\omega_p(\cdot)$, is the oscillation of $s(\cdot,\cdot)$, or $p(\cdot,\cdot)$, near the diagonal region $D:=\{(x,x): x\in \Omega \}$.  Hence, for instance, $\lim_{r\to 0}\omega_p(r)=0$ if and only if $p(\cdot,\cdot)$ is continuous on $D$ uniformly.
We then consider the following logarithmic continuity condition:
\begin{equation}\label{logholder}
 \sup_{0<r\le 1/2 } (\omega_{p}(r) +\omega_{s}(r))\ln \left(\frac{1}{r}\right) \le c_{LH} \quad \text{for some }\ c_{LH}>0.
\end{equation}
Note that this assumption means  $\omega_s(\cdot)$ and $\omega_p(\cdot)$ are so called log-H\"older continuous on $D$.

 The first main result is the local boundedness of the weak solutions to \eqref{mainPDE}.

\begin{theorem}\label{thm1} Let $s,p,K:\R^n\times \R^n \to \R$ be symmetric and satisfy \eqref{s1s2}, \eqref{gamma} and \eqref{Lambda}, respectively. Suppose that for $p(\cdot,\cdot)$, $\lim_{r\to0}\omega_p(r)=0$.   If $u\in \mathbb W^{s(\cdot,\cdot),p(\cdot,\cdot)}(\Omega)\cap L^\infty(\Omega;L^{p(\cdot,\cdot)-1}_{s(\cdot,\cdot)p(\cdot,\cdot)}(\R^n))$ is a weak solution to \eqref{mainPDE}, then $u\in L^\infty_{\mathrm{loc}}(\Omega)$.
\end{theorem}

 The second main result is the H\"older continuity of the weak solutions to \eqref{mainPDE}.

\begin{theorem}\label{mainthm} Let $s,p,K:\R^n\times \R^n \to \R$ be symmetric and satisfy \eqref{s1s2}, \eqref{gamma} and \eqref{Lambda}, respectively.  Suppose that $s(\cdot,\cdot)$ and $p(\cdot,\cdot)$ satisfy \eqref{logholder}. 
If $u\in \mathbb W^{s(\cdot,\cdot),p(\cdot,\cdot)}(\Omega)\cap L^\infty(\Omega;L^{p(\cdot,\cdot)-1}_{s(\cdot,\cdot)p(\cdot,\cdot)}(\R^n))$ is a weak solution to \eqref{mainPDE}, then $u\in C^{0,\alpha}_{\mathrm{loc}}(\Omega)$ for some $\alpha\in (0,1)$ depending on $n$, $s^\pm$, $p^\pm$, $\Lambda$ and $c_{LH}$.
\end{theorem}

There are remarks for the above theorems.

\begin{remark}
We consider the weak solutions belonging to $L^\infty(\Omega;L^{p(\cdot,\cdot)-1}_{s(\cdot,\cdot)p(\cdot,\cdot)}(\R^n))$. This allows to handle  the nonlocal tail given by \eqref{taildef} that is an essential  factor in the analysis of local regularity for nonlocal equations. In Remark~\ref{rmkL}   below, we introduce some examples of functions in  $L^\infty(\Omega;L^{p(\cdot,\cdot)-1}_{s(\cdot,\cdot)p(\cdot,\cdot)}(\R^n))$.
\end{remark}

\begin{remark}
We emphasize that, except the basic assumptions \eqref{s1s2} and \eqref{gamma}, we do not assume any  regularity condition on $s(\cdot,\cdot)$ and $p(\cdot,\cdot)$ in  $(\R^n\times \R^n) \setminus D$. 

For the local boundedness, Theorem~\ref{thm1}, we only assume that the variable exponent $p(\cdot,\cdot)$ is continuous on $D$ uniformly (this can be replaced by a smallness condition on the oscillation of $p(\cdot,\cdot)$, see Remark~\ref{rmksmall}), hence the variable order $s(\cdot,\cdot)$ could be any measurable function satisfying \eqref{s1s2}. On the other hand, for the local H\"older continuity, Theorem~\ref{mainthm}, we assume that   $s(\cdot,\cdot)$ and $p(\cdot,\cdot)$ are log-H\"older continuous on $D$.
\end{remark}

The remaining part of the paper is organized as follows. In Section~\ref{sec2}, we introduce notation and basic tools used in this paper. In Section~\ref{sec3}, we discuss on the existence of weak solution to \eqref{mainPDE}. In Section~\ref{sec4}, we obtain Caccioppoli type and logarithmic estismates. Finally, in Sections~\ref{sec5} and \ref{sec6}, we prove Theorems~\ref{thm1} and \ref{mainthm}, respectively.

\section{\bf Preliminaries}\label{sec2}

\subsection{Notation and function spaces}
For $q\ge 1$ and $t\in (0,1]$ with $tq<n$, $q^*_t:=\frac{nq}{n-tq}$. Let $\Omega\subset\R^n$ be open and bounded, and  let $B_r(x_0)$ be a standard open ball in $\R^n$ centered at $x_0\in \R^n$ with radius $r>0$. If the center $x_0$ is not important we write  $B_r=B_r(x_0)$. For a measurable function $v$ in $\Omega$, $v^\pm:=\max\{\pm v,0\}$, and  $(v)_\Omega$ is denoted  by the mean of $v$ in $\Omega$, that is, $(v)_\Omega:= \fint_{\Omega} v \,dx =\frac{1}{|\Omega|}\int_{\Omega} v\, dx$.

We always assume that $s,p,K: \R^n\times \R^n \to \R$ are symmetric and satisfy \eqref{s1s2}, \eqref{gamma} and \eqref{Lambda}, respectively. For $E\subset \R^n$, we define 
\[
s^-_{E}:=\inf_{(x,y)\in E\times E} s(x,y)
\quad \text{and} \quad 
s^+_{E}:=\sup_{(x,y)\in E\times E} s(x,y),
\]
and $p^\pm_{E}$ by the same quantities for $p(\cdot,\cdot)$. Hence we have  $p^{\pm}_{\R^n}=p^{\pm}$ and  $s^{\pm}_{\R^n}=s^{\pm}$.  For a measurable function $v$ in $\Omega$, set
 \[
\varrho_{s(\cdot,\cdot),p(\cdot,\cdot)}(v;\Omega):= \int_\Omega \int_{\Omega} \frac{|v(x)-v(y)|^{p(x,y)}}{|x-y|^{n+s(x,y)p(x,y)}} \,dy\, dx.
\] 
For  constant  $p\in [1,\infty)$,  we define fractional Sobolev space
$W^{s(\cdot,\cdot),p}(\Omega):=\{v\in L^{p}(\Omega): \varrho_{s(\cdot,\cdot),p}(v;\Omega)<\infty\}$. Note that if $s(\cdot,\cdot)$ is constant with  $s(\cdot,\cdot)\equiv s$, $W^{s,p}(\Omega)$ is the usual fractional Sobolev space.

We next introduce function spaces related to the weak solution to \eqref{mainPDE}. We define $\mathbb W^{s(\cdot,\cdot),p(\cdot,\cdot)}(\Omega)$ by the set of all measurable functions $v:\R^n \to \R$ satisfying that
\[
v|_{\Omega}\in L^{p^{-}_{\Omega}}(\Omega) 
\quad \text{and}\quad
\iint_{\mathcal C_\Omega} \frac{|v(x)-v(y)|^{p(x,y)}}{|x-y|^{n+s(x,y)p(x,y)}} \,dy\, dx  < \infty.
\]
Clearly, $\varrho_{s(\cdot,\cdot),p(\cdot,\cdot)}(u;\Omega) < \infty$ if $v\in \mathbb W^{s(\cdot,\cdot),p(\cdot,\cdot)}(\Omega)$.
We next define the tail space with variable powers $ L^\infty(\Omega;L^{p(\cdot,\cdot)-1}_{s(\cdot,\cdot)p(\cdot,\cdot)}(\R^n))$  by the set of all measurable functions $v:\R^n \to \R$ satisfying that
\[
[v]_{L^\infty(\Omega;L^{p(\cdot,\cdot)-1}_{s(\cdot,\cdot)p(\cdot,\cdot)}(\R^n))}:=\underset{x\in 
\Omega}{\mathrm{ess\, sup}}\int_{\R^n}\frac{|v(y)|^{p(x,y)-1}}{(1+|y|)^{n+s(x,y)p(x,y)}}\,dy <\infty.
\]
Note that  if $v\in   L^\infty(\Omega;L^{p(\cdot,\cdot)-1}_{s(\cdot,\cdot)p(\cdot,\cdot)}(\R^n))$, the following quantity is always finite whenever $x_0\in \R^n$, $r>0$ and $B_\rho(x_0)\subset \Omega$ :
\begin{equation}\label{taildef}
T (v;x_0,r, \rho) := \sup_{x\in B_{\rho}(x_0) }\int_{\R^n\setminus B_r(x_0)}\frac{|v(y)|^{p(x,y)-1}}{|y-x_0|^{n+s(x,y)p(x,y)}}\,dy,
\end{equation}
that is  called the nonlocal tail with variable powers.  Indeed,  since $1+|y| \le (1+\tfrac{|x_0|+1}{r}) |y-x_0|$ for $y\in \R^n\setminus B_{r}(x_0)$, we have
\begin{equation}\label{T1}
T (u;x_0,r,\rho) \le \left(1+\tfrac{|x_0|+1}{r}\right)^{n+s^+p^{+}}[v]_{L^\infty(\Omega;L^{p(\cdot,\cdot)-1}_{s(\cdot,\cdot)p(\cdot,\cdot)}(\R^n))} <\infty.
\end{equation}
Note that they are the variable versions of the tail space and the nonlocal tail introduced in \cite{DKP1,DKP2,KKP1,Pal1}.
For simplicity, if $x_0$ is obvious, we write 
$T (v;r,\rho) = T (v;x_0,r,\rho)$.

\begin{remark}\label{rmkL}
If $v\in L^{\gamma}(\R^n)$ for some $\gamma \ge p^{+}-1$, or if $v\in L^{p^{+}-1}(B_{M}(0))$ for some $M>0$ and is bounded in $\R^n\setminus B_{M}(0)$, then $v\in L^\infty(\Omega;L^{p(\cdot,\cdot)-1}_{s(\cdot,\cdot)p(\cdot,\cdot)}(\R^n))$.  

In particular, $v\in L^\infty(\R^n)\subset L^\infty(\Omega;L^{p(\cdot,\cdot)-1}_{s(\cdot,\cdot)p(\cdot,\cdot)}(\R^n))$ for every $\Omega$, and if $p(\cdot,\cdot)$ is constant with $p(\cdot,\cdot) \equiv p $,  by \cite[Proposition 2.1 and Theorem 6.5]{DPV1} we see that 
\[
W^{s(\cdot,\cdot),p}(\R^n) \subset W^{t,p}(\R^n) \subset L^{p^*_t}(\R^n)  \subset L^\infty(\Omega;L^{p-1}_{s(\cdot,\cdot)p}(\R^n)),
\]
where $t\in(0,s^-)$ is an arbitrary number satisfying $tp<n$. 
\end{remark}

%\begin{remark}
%We do not mention fractional Sobolev spaces with variable powers $W^{s(\cdot,\cdot),p(\cdot,\cdot)}(\Omega)$, that can be found in \cite{CGA1}. Because, the spaces of weak solutions, we considered in this paper, are different from those spaces, and  moreover we do not use any functional analysis result (e.g. Sobolev or Poincar\'e's inequality, or compact embedding) on $W^{s(\cdot,\cdot),p(\cdot,\cdot)}(\Omega)$. Instead we use  well-known properties of the fractional Sobolev spaces with constant powers together with Lemma~\ref{lempxpq}.
%\end{remark}

\subsection{Inequalities}

We introduce inequalities used in this paper. 
The first inequality is a general version of the inequality in \cite[Lemma 4.6]{Co1} in the setting with variable powers. Note that  \cite[Lemma 4.6]{Co1} implies the inclusion $W^{s_2,p_2}(\Omega)\subset W^{s_1,p_1}(\Omega)$ with $0<s_1<s_2<1$ and $p_1\leq p_2$. Therefore, the following lemma also implies a similar inclusion for fractional Sobolev spaces with variable powers.   

\begin{lemma}\label{lempxpq}
Let $\Omega\subset \R^n$ be bounded, and $p_1, p_2,s_1,s_2:\Omega\times \Omega \to (0,\infty)$ be symmetric and satisfy that $s_1(x,y)< s_2(x,y)$,  $p_1(x,y) \le p_2(x,y)$, $0< (p_1)^-_\Omega\le (p_1)^-_\Omega<\infty $  and  $d_1\le s_2(x,y)- s_1(x,y) \le d_2$ for all $x,y\in \Omega$ and for some $0<d_1\le d_2$. Then we have 
\[\begin{split}
 \varrho_{s_1(\cdot,\cdot), p_1(\cdot,\cdot)}(v,\Omega)  \le M\left\{\varrho_{s_2(\cdot,\cdot), p_2(\cdot,\cdot)}(v,\Omega)  + \frac{c(n)}{d_1(p_1)^-_\Omega} |\Omega\cap\{v\neq 0\}| \right\},
\end{split}\]
where 
\[
M=\begin{cases}
\mathrm{diam}(\Omega)^{d_1 (p_1)^-_\Omega} \quad \text{if } \  \mathrm{diam}(\Omega)\le 1,\\
\mathrm{diam}(\Omega)^{d_2 (p_1)^+_\Omega} \quad \text{if } \  \mathrm{diam}(\Omega)> 1,
\end{cases}
\]
and the constant $c(n)>0$ depends only on $n$.
\end{lemma}

\begin{proof} %The inequality follows direct computation.
The direct computation yields that
\[\begin{split}
& \varrho_{s_1(\cdot,\cdot),p_1(\cdot,\cdot)}(v,\Omega)  = \int_\Omega \int_\Omega \left(\frac{|v(x)-v(y)|}{|x-y|^{s_2(x,y)}}\right)^{ p_1(x,y)} \frac{dy\,dx}{|x-y|^{n-(s_2(x,y)- s_1(x,y))p_1(x,y)} } \\
&\le \int_\Omega \int_\Omega \left(\frac{|v(x)-v(y)|}{|x-y|^{s_2(x,y)}}\right)^{ p_2(x,y)} \frac{dy\,dx}{|x-y|^{n-(s_2(x,y)-s_1(x,y)) p_1(x,y)} } \\
&\qquad +2\int_{\Omega\cap\{v\neq 0\}} \int_\Omega   \frac{dy\,dx}{|x-y|^{n-(s_2(x,y)- s_1(x,y))\ p_1(x,y)} } \\
&\le M  \varrho_{s_2(\cdot,\cdot), p_2(\cdot,\cdot)}(v,B_r) +2\int_{\Omega\cap\{v\neq 0\}} \int_{\Omega}  \frac{dy\,dx}{|x-y|^{n-(s_2(x,y)-s_1(x,y)) p_1(x,y)} }.
\end{split}\] 
Note that if $\mathrm{diam}(\Omega)\le 1$
\[
\int_{\Omega} \, \frac{dy}{|x-y|^{n-(s_2(x,y)- s_1(x,y)) p_1(x,y)} } \le  c(n)\int_0^{\mathrm{diam}(\Omega)} r^{-1+d_1(p_1)^-_\Omega }dr = \frac{c(n)}{d_1\tilde p^-_\Omega} M,
\]
and if $\mathrm{diam}(\Omega)> 1$
\[\begin{split}
\int_{\Omega}  \frac{dy}{|x-y|^{n-(s_2(x,y)- s_1(x,y))\tilde p_2(x,y)} } & \le c(n)\left(\int^{\mathrm{diam}(\Omega)}_1 r^{-1+d_2(p_1)^+_\Omega} dr  + \int^1_0 r^{-1+d_1(p_1)^-_\Omega} dr \right)\\
& \le\frac{c(n)}{d_1 (p_1)^-_\Omega} M. \qedhere
\end{split}\]

\end{proof}

We next introduce a fraction Sobolev-Poincar\'e type inequality for the fractional Sobolev spaces with constant powers. This is a simple corollary of \cite[Theorem 6.7]{DPV1}. However, for the sake of completeness, we report a proof.

\begin{lemma}\label{lemSPineq}
Let $0<s<1$, $1\le p<\infty$, and $sp < n$. For any $v\in  L^1(B_r)$ with  $\varrho_{s,p}(v;B_r)<\infty$, we have
\[
 \left(\fint_{B_r}|v-(v)_{B_r}|^{p_s^*}\,dx\right)^{\frac{p}{p_s^*}}  \le c r^{sp}\fint_{B_r}\int_{B_r}\frac{|v(x)-v(y)|^{p}}{|x-y|^{n+sp}}\,dy\,dx,
\]
where $c=c(n,s,p)>0$. Moreover, if $s<t$ and $p<q$ we also have that
\[
 \left(\fint_{B_r}|v-(v)_{B_r}|^{p_s^*}\,dx\right)^{\frac{q}{p_s^*}}  \le c r^{tq}\fint_{B_r}\int_{B_r}\frac{|v(x)-v(y)|^{q}}{|x-y|^{n+tq}}\, dy \,dx
\]
for some $c=c(n,p,q,s,t)>0$, if the right hand side is finite.
\end{lemma}

\begin{proof}
We first observe from H\"older's inequality that 
\[\begin{split}
\fint_{B_r}|v-(v)_{B_r}|^{p}\,dx & \le \frac{1}{|B_r|}\fint_{B_r}\int_{B_r} |v(x)-v(y)|^{p}\,dy\, dx \\
& \le \frac{(2r)^{n+sp}}{|B_r|} \fint_{B_r}\int_{B_r} \frac{|v(x)-v(y)|^{p}}{|x-y|^{n+sp}}\,dy\, dx\\
& \le c r^{sp} \fint_{B_r}\int_{B_r} \frac{|v(x)-v(y)|^{p}}{|x-y|^{n+sp}}\,dy\, dx,
\end{split}\]
which is in fact the fractional version of Poincar\'e inequality. Hence  $v-(v)_{B_r}\in W^{s,p}(B_r)$. Then, by the Sobolev inequality for the fractional Sobolev spaces with constant powers, see for instance \cite[Theorems 6.7 and 6.10]{DPV1},  we have 
\[\begin{aligned}
 \left(\fint_{B_r}|v-(v)_{B_r}|^{p_s^*}\,dx\right)^{\frac{p}{p_s^*}}  & \le c r^{sp} \fint_{B_r}\int_{B_r}\frac{|v(x)-v(y)|^{p}}{|x-y|^{n+sp}}\,dy\,dx + c \fint_{B_r}|v-(v)_{B_r}|^{p}\,dx\\
 & \le c r^{sp} \fint_{B_r}\int_{B_r} \frac{|v(x)-v(y)|^{p}}{|x-y|^{n+sp}}\,dy\, dx.
\end{aligned}\]
Therefore, we have the first inequality. The second inequality is implied by the first inequality with the following estimate: 
\[\begin{split}
&\fint_{B_r}\int_{B_r} \frac{|v(x)-v(y)|^{p}}{|x-y|^{n+sp}}\,dy\, dx\\ 
&\le \left(\fint_{B_r}\int_{B_r} \frac{|v(x)-v(y)|^{q}}{|x-y|^{n+tq}}\,dy\, dx\right)^{\frac{p}{q}} \left(\fint_{B_r}\int_{B_r} \frac{1}{|x-y|^{n+\frac{(s-t)pq}{q-p}}}\,dy\, dx\right)^{\frac{q-p}{q}}\\
& \le c(n,p,q,s,t)\left(r^{(s-t)q}\fint_{B_r}\int_{B_r} \frac{|v(x)-v(y)|^{q}}{|x-y|^{n+tq}}\,dy\, dx\right)^{\frac{p}{q}}.
\end{split}\]
\end{proof}

The next inequality  can be found in \cite[Lemma 4.3]{Co1}.
\begin{lemma} \label{lemineq1}
Let $p>1$ and $a\ge b \ge 0$. For any $\epsilon>0$
\[
a^p - b^p \le \epsilon a^p + \left(\frac{p-1}{\epsilon}\right)^{p-1} (a-b)^p,
\]
hence 
\[
a^p - b^p \le  \epsilon b^p + \left(\epsilon + c\epsilon^{1-p}\right) (a-b)^p
\]
for some $c>0$ depending only on $p$.
\end{lemma}

We end the subsection recalling \cite[Chapter 2. Lemma 4.7]{LU1}.

\begin{lemma}\label{lemseq}
Let $\{y_i\}_{i=0}^\infty$ be a sequence of nonnegative numbers and satisfy
$$
y_{i+1}\leq b_1b_2^{i}y_i^{1+\beta}, \ \ \ i=0,1,2,\dots
$$
for some $b_1,\beta>0$ and $b_2>1$. If 
$$
y_0\leq b_1^{-\frac{1}{\beta}}b_2^{-\frac{1}{\beta^2}},
$$
then $y_i\to 0$ as $i\to \infty$.
\end{lemma}

\section{\bf Existence of weak solution}\label{sec3}

We discuss on the existence of weak solution to \eqref{mainPDE}. This is naturally linked to the existence of a minimizer of
\begin{equation}\label{energy}
\mathcal E (v;\Omega):= \iint_{\mathcal C_\Omega}\frac{1}{p(x,y)} |v(x)-v(y)|^{p(x,y)} K(x,y) \,dy\, dx.
\end{equation}
We say that $u\in \mathbb{W}^{s(\cdot,\cdot),p(\cdot,\cdot)}(\Omega)$ is a minimizer of the energy functional \eqref{energy} if 
\[
\mathcal E(u;\Omega) \le \mathcal E(v;\Omega)
\]
for every $v\in \mathbb{W}^{s(\cdot,\cdot),p(\cdot,\cdot)}(\Omega)$ with $u=v$ a.e. in $\R^n\setminus \Omega$. By a standard argument, see for instance the proof of \cite[Theorem 2.3]{DKP2}, we notice that $u\in \mathbb{W}^{s(\cdot,\cdot),p(\cdot,\cdot)}(\Omega)$ is a weak solution to \eqref{mainPDE} if and only if $u$ is a minimizer of \eqref{energy}.  Hence we prove the existence and uniqueness of a minimizer of \eqref{energy}. We emphasize that no regularity condition on $s(\cdot,\cdot)$ and $p(\cdot,\cdot)$ is assumed. On the other hand, $\Omega$ is assumed to be Lipschitz, in order to apply the compact embedding theorem for the fractional Sobolev space $W^{s,p}$.

\begin{theorem} \label{thmexistence}
 Let $s,p,K:\R^n\times \R^n\to \R$ be symmetric and satisfy \eqref{s1s2}, \eqref{gamma} and \eqref{Lambda}, respectively, $\Omega$ be bounded, open and Lipschitz, and $g\in \mathbb{W}^{s(\cdot,\cdot),p(\cdot,\cdot)}(\Omega)$. There exists a unique minimizer $u\in \mathbb{W}^{s(\cdot,\cdot),p(\cdot,\cdot)}(\Omega)$ with $u=g$ a.e. in $\R^n\setminus \Omega$ of \eqref{energy}. 

Moreover, suppose that $p^+-1 < \frac{np^-}{n-s^-p^-}$ or $ s^-p^-\ge n$. If $g\in \mathbb{W}^{s(\cdot,\cdot),p(\cdot,\cdot)}(\Omega)\cap L^\infty(\Omega;L^{p(\cdot,\cdot)-1}_{s(\cdot,\cdot)p(\cdot,\cdot)}(\R^n))$, then %the minimizer $u\in \mathbb{W}^{s(\cdot,\cdot),p(\cdot,\cdot)}(\Omega)\cap L^\infty(\Omega;L^{p(\cdot,\cdot)-1}_{s(\cdot,\cdot)p(\cdot,\cdot)}(\R^n))$.
 the minimizer $u\in \mathbb{W}^{s(\cdot,\cdot),p(\cdot,\cdot)}(\Omega)$ is in $ L^\infty(\Omega;L^{p(\cdot,\cdot)-1}_{s(\cdot,\cdot)p(\cdot,\cdot)}(\R^n))$.
\end{theorem}

\begin{proof}   
We follow the argument in the proof of  \cite[Theorem 2.3]{DKP2}. Note that the uniqueness is the  direct consequence of the strict convexity of the mapping $t\mapsto t^{p(x,y)}$ which is uniform in $(x,y)\in\R^n\times\R^n$, since $p^->1$. 

We now prove the existence. 
  Let $\{u_k\}_{k=1}^\infty\subset\mathbb W^{s(\cdot,\cdot),p(\cdot,\cdot)}(\Omega)$ be a minimizing sequence with $u_k=g$ a.e. in $\R^n\setminus \Omega$. (Note that by the definition of $g$ the admissible set  of the energy functional, that is $\mathbb W^{s(\cdot,\cdot),p(\cdot,\cdot)}_g(\Omega)=\{v\in \mathbb W^{s(\cdot,\cdot),p(\cdot,\cdot)}(\Omega) : v=g\ \ \text{a.e. in }\R^n\setminus \Omega \}$, is non-empty.) Then there exists $M>0$ such that
\[
\varrho_{s(\cdot,\cdot),p(\cdot,\cdot)}(u_k;\Omega) \le \iint_{\mathcal C_\Omega}\frac{|u_k(x)-u_k(y)|^{p(x,y)}}{|x-y|^{n+s(x,y)p(x,y)}}\,dydx \le M \quad \text{for all }\  k=1,2,\dots.
\] 
Hence, by Lemma~\ref{lempxpq}, we have that for any $t\in (0,s^-)$,  $\varrho_{t,p^-}(u_k;\Omega)$ is bounded for $k$.
Moreover, since $u_k-g=0$ a.e. in $\R^n\setminus \Omega$, by using the fractional Sobolev inequality \cite[Theorem 6.5]{DPV1} we have that  for some $B_R=B_R(0)\supset \Omega$ with $R\ge 1$ and $t\in(0,\frac{s^-}{2})$ with $\frac{np^-}{n+tp^-}=:q>1$, 
\[\begin{split}
&\left(\int_{B_R}  |u_k-g|^{p^-}\,dx\right)^{\frac{q}{p^-}}  \le  c\varrho_{t,q}(u_k-g;\R^n)\\
&\le c\varrho_{t,q}(u_k-g;B_R)+ c \int_{ B_R} |u_k(y)-g(y)|^{q} \left(\int_{\R^n \setminus B_R}\frac{1}{|x-y|^{n+tq}}\,dx\right)\,dy \\
&\le c\varrho_{t,q}(u_k-g;B_R)+ c \int_{ B_R} |u_k(y)-g(y)|^{q} \left(\int_{B_{2R} \setminus B_R}\frac{1}{|x-y|^{n+tq}}\,dx\right)\,dy\\
&\le c \int_{B_{2R} }\int_{B_R}\frac{|(u_k(x)-g(x))-(u_k(y)-g(y))|^{q}}{|x-y|^{n+tq}}\,dx\,dy \le c\varrho_{t,q}(u_k-g;B_{2R}) .
\end{split}\]
Note that in the second inequality we used the fact that  
\[
\int_{\R^n \setminus B_R}\frac{1}{|x-y|^{n+tq}}\,dx \le \frac{c(n)}{(r-|y|)^{tq}} \le c(n) \int_{B_{2R} \setminus B_R}\frac{1}{|x-y|^{n+tq}}\,dx \quad \text{for all }\ y\in B_{R}.
\]
Again using Lemma~\ref{lempxpq}, 
we have that
\[\begin{split}
&\left(\int_{\Omega}  |u_k-g|^{p^-}\,dx \right)^{\frac{q}{p^-}}  = \left(\int_{B_R}  |u_k-g|^{p^-}\,dx \right)^{\frac{q}{p^-}}\\
& \le  c R^{s^+p^-} (\varrho_{s(\cdot,\cdot),p(\cdot,\cdot)}(u_k-g;B_{2R})+ R^n)\\
& \le  c R^{s^+p^-} \left(\iint_{\mathcal C_\Omega}\frac{|u_k(x)-u_k(y)|^{p(x,y)}}{|x-y|^{n+s(x,y)p(x,y)}}\,dydx+\iint_{\mathcal C_\Omega}\frac{|g(x)-g(y)|^{p(x,y)}}{|x-y|^{n+s(x,y)p(x,y)}}\,dydx + R^n\right)\\
&\le  c R^{s^+p^-} \left(M+\iint_{\mathcal C_\Omega}\frac{|g(x)-g(y)|^{p(x,y)}}{|x-y|^{n+s(x,y)p(x,y)}}\,dydx + R^n\right) \quad \text{for all }\ k=1,2,\dots,
\end{split}\]
which implies that $\{u_k\}$ is bounded in $L^{p^-}(\Omega)$, hence $\{u_k\}$ is bounded in $W^{t,p^-}(\Omega)$. 

By the compact embedding theorem for the fractional Sobolev space $W^{t,p^-}(\Omega)$ \cite[Theorem 7.1]{DPV1}, there exist a subsequence $\{u_{k_j}\}_{j=1}^\infty$ and   $u\in L^{p^-}(\Omega)$ such that 
\[
u_{k_j}\ \longrightarrow \ u \quad\text{as} \quad j\to \infty \quad \text{strongly in } \ L^{p^-}(\Omega)  
\]
and
\[
u_{k_j}\ \longrightarrow \ u \quad\text{as} \quad j\to \infty \quad \text{a.e. in } \ \Omega. 
\] 
We extend $u$ by $g$ in $\R^n\setminus\Omega$. Then Fatou's lemma yields 
\[\begin{split}
\iint_{\mathcal C_\Omega}&\frac{1}{p(x,y)}|u(x)-u(y)|^{p(x,y)}K(x,y)\,dydx \\
&\le \liminf_{j\to\infty}\iint_{\mathcal C_\Omega}\frac{1}{p(x,y)}|u_{k_j}(x)-u_{k_j}(y)|^{p(x,y)}K(x,y)\,dydx,
\end{split}\]
that means $u\in \mathbb W^{s(\cdot,\cdot),p(\cdot,\cdot)}(\Omega)$, and it is the minimizer.

Next we further assume $g\in \mathbb{W}^{s(\cdot,\cdot),p(\cdot,\cdot)}(\Omega)\cap L^\infty(\Omega;L^{p(\cdot,\cdot)-1}_{s(\cdot,\cdot)p(\cdot,\cdot)}(\R^n))$, and $p^+-1 < \frac{np^-}{n-s^-p^-}$ or $ s^-p^-\ge n$. Then one can find $t\in (0,s^-)$ such that $tp^-<n$ and $p^+-1 \le \frac{np^-}{n-tp^-}$. Then $u\in W^{t,p^-}(\Omega)$  by Lemma~\ref{lempxpq}, where $u\in \mathbb W^{s(\cdot,\cdot),p(\cdot,\cdot)}(\Omega)$ is the unique minimizer.   Therefore, by the embedding theorem for the fractional Sobolev space \cite[Theorem 6.7]{DPV1} we have that $u\in L^{\frac{np^-}{n-tp^-}}(\Omega)\subset L^{p^+-1}(\Omega)$. This and the fact that $u=g$ a.e. in $\R^n\setminus \Omega$ imply that the minimizer $u$ is in $L^\infty(\Omega;L^{p(\cdot,\cdot)-1}_{s(\cdot,\cdot)p(\cdot,\cdot)}(\R^n))$. 
\end{proof}

\section{\bf Caccioppoli and logarithmic estimates} \label{sec4}

We obtain a Caccioppoli type estimate and a logarithmic estimate for the weak solutions to \eqref{mainPDE}.  In order to obtain a Caccioppoli type estimate for a weak solution to a local/nonlocal equation, we take a multiplicative function of the weak solution and  a cut-off function as a test function in the weak form of the problem. However, in the variable exponent case, that is, $p(\cdot,\cdot)$ is not constant, it is unclear that the multiplicative function can be taken as a test function. More precisely, even if $v\in \mathbb W^{s(\cdot,\cdot),p(\cdot,\cdot)}(\Omega)$, we couldn't prove that $v\eta \in W^{s(\cdot,\cdot),p(\cdot,\cdot)}(\Omega)$, where $\eta$ is a standard cut-off function in a ball contained in $\Omega$. (Note that if $p^-$ and $p^+$ are sufficiently close, it is correct.) To overcome this problem, we will use an approximation argument together with the following lemma.

\begin{lemma}\label{lemueta}
Let $v\in  L^{p^+}(B_{2r})$ satisfy $\varrho_{s(\cdot,\cdot),p(\cdot,\cdot)}(v,B_{2r})<\infty$, and $\eta\in W^{1,\infty}_0(B_r)$. Then $ \varrho_{s(\cdot,\cdot),p(\cdot,\cdot)}(v\eta,\R^n)<\infty$. In particular, $v\eta \in \mathbb W^{s(\cdot,\cdot),p(\cdot,\cdot)}(\Omega)$ whenever $\Omega\supset B_{2r}$.
\end{lemma}
\begin{proof}
We have to show that
\[
\varrho_{s(\cdot,\cdot),p(\cdot,\cdot)}(v\eta,\R^n) =\varrho_{s(\cdot,\cdot),p(\cdot,\cdot)}(v\eta,B_{2r}) +2\int_{B_{2r}} \int_{\R^n\setminus B_{2r}} \frac{|v(x)\eta(x)|^{p(x,y)}}{|x-y|^{n+s(x,y) p(x,y)}} \, dy\,dx <\infty.
\]
We first estimate the first term. Using  the Mean Value Theorem, \eqref{s1s2} and \eqref{gamma},
\[\begin{aligned}
&\varrho_{s(\cdot,\cdot),p(\cdot,\cdot)}(v\eta,B_{2r}) \\
&\le c \int_{B_{2r}} \int_{B_{2r}} \frac{|(v(x)-v(y))\eta(y)|^{p(x,y)}}{|x-y|^{n+s(x,y) p(x,y)}} \, dy\,dx+c\int_{B_{2r}} \int_{B_{2r}} \frac{|v(x)(\eta(x)-\eta(y))|^{p(x,y)}}{|x-y|^{n+s(x,y) p(x,y)}} \, dy\,dx \\
&\le c(\|\eta\|_{L^\infty(B_r)}^{p^{+}} +1)\varrho_{s(\cdot,\cdot),p(\cdot,\cdot)}(v,B_{2r})\\
&\qquad+c (\|D\eta\|_{L^\infty(B_r)}^{p^{+}}+1) \int_{B_{2r}} (|v(x)|^{p^+}+1) \int_{B_{4r}(x)} |x-y|^{-n+(1-s(x,y))p(x,y)} \, dy\,dx\\
&\le c(\|\eta\|_{L^\infty(B_r)}^{p^{+}} +1)\varrho_{s(\cdot,\cdot),p(\cdot,\cdot)}(v,B_{2r})\\
&\qquad +c \max\{r^{(1-s^{+})p^{-}},r^{(1-s^{-})p^{+}}\}(\|D\eta\|_{\infty}^{p^{+}}+1) \bigg(\int_{B_r} |v(x)|^{p+}\,dx +r^n\bigg)\\
&<\infty.
\end{aligned}\]
We next estimate the second term. Using the fact that $\eta \equiv 0$ in $\R^n\setminus B_r$, \eqref{s1s2} and \eqref{gamma},
\[\begin{aligned}
&\int_{B_{2r}} \int_{\R^n\setminus B_{2r}} \frac{|v(x)\eta(x)|^{p(x,y)}}{|x-y|^{n+s(x,y) p(x,y)}} \, dy\,dx = (\|\eta\|_\infty^{p^{+}} +1) \int_{B_{r}} \int_{\R^n\setminus B_{2r}} \frac{|v(x)|^{p^+}+1}{|x-y|^{n+s(x,y) p(x,y)}} \, dy\,dx\\
& \le (\|\eta\|_{L^\infty(B_r)}^{p^{+}} +1) \int_{B_{r}} (|v(x)|^{p^+}+1) \int_{\R^n\setminus B_{r}(x)} |x-y|^{-n-s(x,y) p(x,y)} \, dy\,dx\\
& \le c \max\{r^{-s^-p^-},r^{-s^+p^+}\} (\|\eta\|_\infty^{p^{+}} +1) \left(\int_{B_{r}} |v(x)|^{p^+} \,dx+r^n\right) <\infty.
\end{aligned}\]

\end{proof}

Now we  obtain a Caccioppoli type inequality.

\begin{lemma}\label{lemcaccio}
Let $u\in \mathbb W^{s(\cdot,\cdot),p(\cdot,\cdot)}(\Omega)\cap L^\infty(\Omega;L^{p(\cdot,\cdot)-1}_{s(\cdot,\cdot)p(\cdot,\cdot)}(\R^n))$ be a weak solution to \eqref{mainPDE}, and let $B_{2r}=B_{2r}(x_0) \Subset\Omega$ satisfying 
\begin{equation}\label{Rselect2}
r\le \frac{1}{2} \quad \text{and}\quad \omega_{p}(r)\le \frac{(1-s^+)p^-}{2s^+}.
\end{equation}
Then for any $0<\rho<r$ and any $k\in \R$, we have
\begin{equation}\label{caccio}\begin{split}
 &\varrho_{s(\cdot,\cdot),p(\cdot,\cdot)}((u-k)_\pm,B_{\rho})  + \int_{B_{\rho}}(u(x)-k)_\pm\bigg(\int_{\R^n}\frac{(k-u(y))_\pm^{p(x,y)-1}}{|x-y|^{n+s(x,y)p(x,y)}}\,dy\bigg)\,dx  \\
&\le \frac{1}{(r-\rho)^{p_1}} \int_{B_{r}} \int_{B_{r}}  \frac{\max\{(u(x)-k)_\pm,(u(y)-k)_\pm\}^{p(x,y)}}{|x-y|^{n+s(x,y)p(x,y)-p_1}}\,dy\, dx\\
&\qquad + c  \left(\frac{r}{r-\rho}\right)^{n+s_2p_2}  \int_{B_r}  (u(x)-k)_\pm \bigg[\int_{\R^n\setminus B_{r}}\frac{(u(y)-k)_\pm^{p(x,y)-1}}{|y-x_0|^{n+s(x,y)p(x,y)}}\,dy \bigg] \, dx\\
& \le c \frac{r^{p_2-s_2p_2}}{(r-\rho)^{p_2}} \int_{A^\pm(k,r)} \Big((u(x)-k)_\pm^{p_2}+1\Big) \,dx\\
&\qquad + c \frac{r^{n+s_2p_2}}{(r-\rho)^{n+s_2p_2}}   T((u-k)_\pm,x_0,r,r)\int_{B_{r}}(u-k)_\pm \,dx ,
\end{split}\end{equation}
where 
\[
s_1:=s^-_{B_r}, \quad s_2:=s^+_{B_r}, \quad  p_1:=p^-_{B_r} \quad p_2:=p^+_{B_r},
\]
\[
A^\pm(k,r) =A^\pm(k,r,x_0) := \{x\in B_r: (u(x)-k)_\pm>0\},
\]
and the constant $c>0$ depends only on $n$, $s^\pm$, $p^\pm$ and $\Lambda$.
\end{lemma}

\begin{proof}
We prove the desired estimate only for  ``$+$" sign. Then the estimate for ``$-$" sign directly follows by considering $-u$ that is also a weak solution to \eqref{mainPDE}.  
Let  $\eta\in C^\infty_c(B_{\frac{r+\rho}{2}})$ with $0\le \eta \le 1$, $\eta\equiv 1$ in $B_{\rho}$ and $|\n \eta| \le 4/(r-\rho)$ in $\R^n$. Define $w_{\pm}:=(u-k)_{\pm}$ and $w^l_+:=(\min\{u,l\}-k)_{+}$, where $l>k$. Note that $w^l_+(x)\le w_+(x)$ and $|w^l_+(x)-w^l_+(y)|\le |w_+(x)-w_+(y)|$ for every $x,y\in \R^n$ and every $l>k$. By Lemma~\ref{lemueta}, $w^l_+\eta^{p_2} \in \mathbb W^{s(\cdot,\cdot),p(\cdot,\cdot)}(\Omega)$ for every $l>k$ since $0\le w^l_+\le {l-k}$ and $\varrho_{s(\cdot,\cdot),p(\cdot,\cdot)}(w^l_+,B_{2r})\le \varrho_{s(\cdot,\cdot),p(\cdot,\cdot)}(w_+,B_{2r})<\infty$. Hence  we take  $w^l_+\eta^{p_2}$ as a test function $\phi$ in \eqref{weakform}, and obtain an estimate for $w^l_k$ first. Then by taking the limit as $l\to\infty$, we obtain the desired estimate.

Recalling  \eqref{weakform} with $\phi=w_+^l\eta^{p_2} $,
\[\begin{split}
0&= \int_{B_{r}}\int_{B_{r}} |u(x)-u(y)|^{p(x,y)-2}(u(x)-u(y)) (\phi(x)-\phi(y))K(x,y)\,dy\, dx \\
&\quad  + 2 \int_{B_r}\int_{\R^n\setminus B_{r}} |u(x)-u(y)|^{p(x,y)-2}(u(x)-u(y)) (\phi(x)-\phi(y))K(x,y)\,dy\, dx\\
&= : I_1 +2 I_2.
\end{split}\]
We first estimate $I_1$. We first have  
\begin{equation}\label{eq1}
|u(x)-u(y)|^{p(x,y)-2}(u(x)-u(y)) (\phi(x)-\phi(y))=0, \quad \text{ for  }\  x,y\in B_r\setminus A^+(k,r).
\end{equation}
If $x\in A^+(k,r)$ and $y\in B_{r}\setminus A^+(k,r)$, then   $u(x)> k \ge u(y)$ hence
\begin{equation}\label{eq2}\begin{aligned}
&|u(x)-u(y)|^{p(x,y)-2}(u(x)-u(y))  (\phi(x)-\phi(y))\\
&\ge (w^l_+(x)+w_-(y))^{p(x,y)-1} w^l_+(x)\eta(x)^{p_2}\\
&\ge \frac{1}{2}| w^l_+(x)-w^l_+(y)|^{p(x,y)}\eta(x)^{p_2 }+\frac{1}{2} w_-(y)^{p(x,y)-1}w^l_+(x)\eta(x)^{p_2 }.
\end{aligned}\end{equation}
Now we suppose $x,y\in A^+(k,r)$. If  $u(x)> u(y)>k$ and $\eta(x)^{p_2}\ge \eta(y)^{p_2}$, 
\[\begin{split}
&|u(x)-u(y)|^{p(x,y)-2}(u(x)-u(y)) (\phi(x)-\phi(y))\\
& =|w_+(x)-w_+(y)|^{p(x,y)-1}(w^l_+(x)\eta(x)^{p_2}-w^l_+(y)\eta(y)^{p_2})\\
&\ge |w^l_+(x)-w^l_+(y)|^{p(x,y)}\eta(x)^{p_2}.
\end{split}\]
If $u(x)> u(y)>k$ and $\eta(x)^{p_2}< \eta(y)^{p_2}$,
\[\begin{split}
&|u(x)-u(y)|^{p(x,y)-2}(u(x)-u(y)) (\phi(x)-\phi(y))\\
& \ge (w^l_+(x)-w^l_+(y))^{p(x,y)}\eta(y)^{p_2} - (w_+(x)-w_+(y))^{p(x,y)-1}w^l_+(x) (\eta(y)^{p_2}-\eta(x)^{p_2}).
\end{split}\]
Applying Lemma~\ref{lemineq1} with $p=p_1$, $a=\eta(y)^{\frac{p_2}{p_1}}$, $b=\eta(x)^{\frac{p_2}{p_1}}$,  and $\epsilon=\frac{1}{2}\frac{w_+(x)-w_+(y)}{w_+(x)}$, and using the facts that $0\le \eta\le 1$ and $| \eta(y)- \eta(x)|\le \|\n \eta\|_{L^\infty(B_r)}|x-y|\le c\frac{|x-y|}{r-\rho}$,
\[\begin{split}
\eta(y)^{p_2}-\eta(x)^{p_2} & \le \frac{1}{2} \frac{w_+(x)-w_+(y)}{w_+(x)} \eta(y)^{p_2} +\left\{\frac{(p_1-1)w_+(x)}{w_+(x)-w_+(y)}\right\}^{p_1-1} \left(\eta(y)^{\frac{p_2}{p_1}}-\eta(x)^{\frac{p_2}{p_1}}\right)^{p_1}\\
& \le \frac{1}{2}\frac{w_+(x)-w_+(y)}{w_+(x)}  \eta(y)^{p_2} +c \left(\frac{w_+(x)}{w_+(x)-w_+(y)}\right)^{p(x,y)-1} \frac{|x-y|^{p_1}}{(r-\rho)^{p_1}},
\end{split}\]
hence
\[\begin{split}
&|u(x)-u(y)|^{p(x,y)-2}(u(x)-u(y)) (\phi(x)-\phi(y))\\
& \ge \frac12 |w^l_+(x)-w^l_+(y)|^{p(x,y)}\eta(y)^{p_2} - c w_+(x)^{p(x,y)} \frac{|x-y|^{p_1}}{(r-\rho)^{p_1}}\\
& \qquad - \frac{1}{2}\left(|w_+(x)-w_+(y)|^{p(x,y)} -|w^l_+(x)-w^l_+(y)|^{p(x,y)}\right)\eta(y)^{p_2}.
\end{split}\]
Therefore, by considering the symmetry of  $p(\cdot,\cdot)$ and $s(\cdot,\cdot)$, we have that for every $x,y\in A^+(k,r)$,
\begin{equation}\label{eq3}\begin{split}
&|u(x)-u(y)|^{p(x,y)-2}(u(x)-u(y)) (\phi(x)-\phi(y))\\
& \ge \frac12 |w^l_+(x)-w^l_+(y)|^{p(x,y)}\min\{\eta(x), \eta(y)\}^{p_2}\\
&\qquad - c \max \{w_+(x),w_+(y) \}^{p(x,y)} \frac{|x-y|^{p_1}}{(r-\rho)^{p_1}}\\
& \qquad - \frac{1}{2}\left(|w_+(x)-w_+(y)|^{p(x,y)} -|w^l_+(x)-w^l_+(y)|^{p(x,y)}\right)\max \{ \eta(x),\eta(y)\}^{p_2}.
\end{split}\end{equation}
Combining the results in \eqref{eq1}, \eqref{eq2} and \eqref{eq3}, we get
\[\begin{split}
I_1 &\ge \frac{1}{2} \left(\varrho_{s(\cdot,\cdot),p(\cdot,\cdot)}(w^l_+,B_{\rho})  + \int_{B_{\rho}}w_+(x)\left[\int_{B_{r}}\frac{w_-(y)^{p(x,y)-1}}{|x-y|^{n+s(x,y)p(x,y)}}\,dy\right]\,dx\right) \\
& \qquad - \frac{c}{(r-\rho)^{p_1}} \int_{B_{r}} \int_{B_{r}}  \frac{\max\{w_+(x),w_+(y)\}^{p(x,y)}}{|x-y|^{n+s(x,y)p(x,y)-p_1}}\,dy\, dx\\
& \qquad  - \frac{1}{2}\left(\varrho_{s(\cdot,\cdot),p(\cdot,\cdot)}(w_+,B_{r}) -\varrho_{s(\cdot,\cdot),p(\cdot,\cdot)}(w^l_+,B_{r})\right).
\end{split}\]

On the other hand,
\[\begin{split}
I_2&\ge \int_{A^+(k,\rho)}  w^l_+(x) \bigg[\int_{\{u(x)\ge u(y)\}\setminus B_{r}}\frac{(u(x)-u(y))^{p(x,y)-1}}{|x-y|^{n+s(x,y)p(x,y)}}\,dy \bigg] \, dx\\
&\qquad -  \int_{A^+(k,\frac{r+\rho}{2})}  w_+(x) \left[\int_{\{u(x) < u(y)\}\setminus B_{r}}\frac{(u(y)-u(x))^{p(x,y)-1}}{|x-y|^{n+s(x,y)p(x,y)}}\,dy \right] \, dx\\
&\ge \int_{B_\rho} w^l_+(x) \bigg[\int_{\R^n\setminus B_{r}}\frac{ w_-(y)^{p(x,y)-1}}{|x-y|^{n+s(x,y)p(x,y)}}\,dy \bigg] \, dx\\
&\qquad - c \left(\frac{r}{r-\rho}\right)^{n+s_2p_2}  \int_{B_r}  w_+(x) \bigg[\int_{\R^n\setminus B_{r}}\frac{w_+(y)^{p(x,y)-1}}{|y-x_0|^{n+s(x,y)p(x,y)}}\,dy \bigg] \, dx.
\end{split}\]
Here we  used the fact that 
\[
|x-y| \ge |y-x_0| -|x-x_0| \ge  |y-x_0| - \frac{r+\rho}{2} \frac{|y-x_0|}{r} =\frac{r-\rho}{2r}  |y-x_0|
\]
for every $x\in B_{\frac{r+\rho}{2}}$ and $y\in \R^n\setminus B_{r}$, where $x_0$ is the center of $B_r$.

Therefore, from the above estimates for $I_1$ and $I_2$ with $I_1+2I_2=0$, we have that for every $l>k$,
\[\begin{split}
 &\varrho_{s(\cdot,\cdot),p(\cdot,\cdot)}(w^l_+,B_{\rho})  + \int_{B_{\rho}}w^l_+(x)\bigg[\int_{\R^n}\frac{w_-(y)^{p(x,y)-1}}{|x-y|^{n+s(x,y)p(x,y)}}\,dy\bigg]\,dx  \\
&\le \frac{c}{(r-\rho)^{p_1}} \int_{B_{r}} \int_{B_{r}}  \frac{\max\{w_+(x),w_+(y)\}^{p(x,y)}}{|x-y|^{n+s(x,y)p(x,y)-p_1}}\,dy\, dx\\
&\qquad + c  \left(\frac{r}{r-\rho}\right)^{n+s_2p_2}  \int_{B_r}  w_+(x) \bigg[\int_{\R^n\setminus B_{r}}\frac{w_-(y)^{p(x,y)-1}}{|y-x_0|^{n+s(x,y)p(x,y)}}\,dy \bigg] \, dx\\
&\qquad +c \left(\varrho_{s(\cdot,\cdot),p(\cdot,\cdot)}(w_+,B_{r}) -\varrho_{s(\cdot,\cdot),p(\cdot,\cdot)}(w^l_+,B_{r})\right).
\end{split}\]
Since $w^l_+(x) \nearrow w_+(x)$ and $|w^l_+(x)-w^l_+(y)|\nearrow |w_+(x)-w_+(y)|$ as $l\to \infty$ for a.e. $x,y\in B_r$, by the Lebesgue monotone convergence theorem, sending $l$ in the above estimate to $\infty$, 
we obtain the first inequality in \eqref{caccio}. 

The second inequality in \eqref{caccio} is obtained from the definition the tail and the following estimates: since $p_1-s_2p_2\ge (1-s^+)p^--s^+\omega_{p}(r) \ge \frac{(1-s^+)p^-}{2} >0$ by \eqref{Rselect2},
\[\begin{split}
&\frac{1}{(r-\rho)^{p_1}} \int_{B_{r}} \int_{B_{r}}  \frac{\max\{w_+(x),w_+(y)\}^{p(x,y)}}{|x-y|^{n+s(x,y)p(x,y)-p_1}}\,dy\, dx\\
&\le \frac{r^{p_2}}{r^{p_1}(r-\rho)^{p_2}} \int_{B_{r}} \int_{B_{r}} \frac{\max\{w_+(x),w_+(y)\}^{p(x,y)} }{|x-y|^{n+s_2p_2-p_1}}\,dy\, dx\\
&\le \frac{r^{p_2}}{r^{p_1}(r-\rho)^{p_2}} \int_{A^+(k,r)} \left(w_+(x)^{p_2}+1\right) \int_{B_{2r}(x)} \frac{1}{|x-y|^{n+s_2p_2-p_1}}\,dy\,dx\\
&\le  c \frac{r^{p_2-s_2p_2}}{(r-\rho)^{p_2}} \int_{A^+(k,r)} \left(w_+(x)^{p_2}+1\right) \,dx. 
\end{split}\]
\end{proof}

Next we obtain logarithmic estimates that will be used crucially in the proof of the H\"older regularity in Section~\ref{sec5}. Here we consider a locally bounded and nonnegative solution.

\begin{lemma} \label{lemlog}
Let $u\in\mathbb W^{s(\cdot,\cdot),p(\cdot,\cdot)}(\Omega)\cap L^\infty(\Omega;L^{p(\cdot,\cdot)-1}_{s(\cdot,\cdot)p(\cdot,\cdot)}(\R^n))$ be a weak  solution to \eqref{mainPDE}
such that $u\in L^\infty(\Omega')$ for some $\Omega'\Subset\Omega$. If $r\in(0,1)$ satisfies
%such that $u\ge 0$ in $B_r=B_r(x_0)\Subset \Omega$ with $r>0$ satisfying 
\begin{equation}\label{Rselect5}
%\omega_p(r)\le \min\left\{\frac{s^-(p^-)^2}{4n},\, \frac{\ln2}{\ln(1+\|u\|_{L^\infty(\Omega')})}\right\}}
\omega_p(r)\le \frac{\ln2}{\ln(1+\|u\|_{L^\infty(\Omega')})},
\end{equation}
$u\ge 0$ in  $B_r=B_r(x_0)\subset \Omega'$ and $d>0$, then for any $B_\rho=B_\rho(x_0)$ with $0<\rho< r/2$
\[\begin{split}
 &\int_{B_\rho}\int_{B_\rho} \left|\ln\bigg(\frac{d+u(x)}{d+u(y)}\bigg)\right|^{p_2}K(x,y)\,dx\,dy  \\
&\le c \rho^{n-s_2p_1+p_1-p_2}+ c \max\{d,d^{-1}\}^{p_4-p_3}  \rho^{n-s_4p_4} + \frac{c \rho^n}{d^{p_2-1}}T (u_-+\|u\|_{L^\infty(B_r)},r,2\rho)
\end{split}\]
for some $c=c(n,s^\pm,p^\pm,\Lambda)>0$, where
\begin{equation}\label{si}
s_1:=s^-_{B_{2\rho}},\quad s_2:=s^+_{B_{2\rho}},\quad s_3:=s^-_{B_{r}},\quad s_4:=s^+_{B_{r}},
\end{equation}
\begin{equation}\label{pi}
p_1:=p^-_{B_{2\rho}}, \quad p_2:=p^+_{B_{2\rho}}, \quad p_3:=p^-_{B_{r}},\quad p_4:=p^+_{B_{r}}.
\end{equation}
\end{lemma}

\begin{proof}
Let $\eta \in C^\infty_0(B_{3\rho/2})$ such that
$0\le \eta \le 1$, $\eta\equiv 1$ in $B_\rho$ and $|D\eta|\le c(n) /\rho$.  Note that, since the Mean Value Theorem yields 
$|(u(x)+d)^{1-p_2}-(u(y)+d)^{1-p_2}|\le (p^+-1)\max\{1,d^{-1}\}^{p^+}|u(x)-u(y)|$  and $0<(u+d)^{1-p_2}\le\min\{1,d\}^{1-p^+}$ for every $x,y\in B_{2\rho}$, by Lemma~\ref{lemueta}  we can take $(u+d)^{1-p_2}\eta^{p_2}$ as a test function in \eqref{weakform}. Hence,
\[\begin{split}
0&=\int_{B_{2\rho}}\int_{B_{2\rho}} |u(x)-u(y)|^{p(x,y)-2}(u(x)-u(y))\\
&\hspace{5cm}\times \left[\frac{\eta(x)^{p_2}}{(u(x)+d)^{p_2-1}}-\frac{\eta(y)^{p_2}}{(u(y)+d)^{p_2-1}}\right] K(x,y)\,dx\,dy\\
&\qquad +2 \int_{B_{2\rho}} \int_{\R^n\setminus B_{2\rho}}\frac{ |u(x)-u(y)|^{p(x,y)-2}(u(x)-u(y)) \eta(x)^{p_2}}{(u(x)+d)^{p_2-1}} K(x,y)\,dy\,dx \\
&=: I_1+2I_2.
\end{split}\]

We first estimate $I_1$. Suppose $u(x)>u(y)$ for $x,y\in B_{2\rho}$. By the second inequality in Lemma~\ref{lemineq1} with $a=\eta(x)^\frac{p_2}{p(x,y)}$, $b=\eta(y)^{\frac{p_2}{p(x,y)}}$, $p=p(x,y)$ and $
\epsilon=\delta \frac{u(x)-u(y)}{u(x)+d}\in (0,1)$,
we have that for any  $\delta\in (0,1)$, 
\[\begin{split}
\eta(x)^{p_2} & \le  \eta(y)^{p_2} + \delta \frac{u(x)-u(y)}{u(x)+d} \eta(y)^{p_2} \\
&\qquad +  c \left(\delta \frac{u(x)-u(y)}{u(x)+d}\right)^{1-p(x,y)} |\eta(x)^\frac{p_2}{p(x,y)}-\eta(y)^\frac{p_2}{p(x,y)}|^{p(x,y)}
\end{split}\]
(if $\eta(x)\le \eta(y)$ the inequality is trivial). Using this inequality,
\[\begin{split}
&|u(x)-u(y)|^{p(x,y)-2}(u(x)-u(y))\left[\frac{\eta(x)^{p_2}}{(u(x)+d)^{p_2-1}}-\frac{\eta(y)^{p_2}}{(u(y)+d)^{p_2-1}}\right] K(x,y)\\
&\le \underbrace{\frac{(u(x)-u(y))^{p(x,y)-1}}{(u(x)+d)^{p_2-1}}\eta(y)^{p_2} \left[1+\delta\frac{u(x)-u(y)}{u(x)+d}-\left(\frac{u(x)+d}{u(y)+d}\right)^{p_2-1}\right]K(x,y)}_{=:J}\\
&\qquad+ c \delta^{1-p(x,y)} (u(x)+d)^{p(x,y)-p_2} |\eta(x)^\frac{p_2}{p(x,y)}-\eta(y)^\frac{p_2}{p(x,y)}|^{p(x,y)} K(x,y).
\end{split}\]
Now we estimate 
\[
J= \frac{(u(x)-u(y))^{p(x,y)}}{(u(x)+d)^{p_2}}\eta(y)^{p_2} \left[\delta+\tfrac{1-(\frac{u(y)+d}{u(x)+d})^{1-p_2}}{1-\frac{u(y)+d}{u(x)+d}}\right]K(x,y).
\]
Then we consider $g(t):= \frac{1-t^{1-p}}{1-t}$, where $p>1$ and $t\in(0,1)$. Note that with constant $p>1$
\begin{equation}\label{logpf0}
\frac{1-t^{1-p}}{1-t}\le -(p-1)\quad ^\forall t\in(0,1) \quad\text{and}\quad  \frac{1-t^{1-p}}{1-t}\le -\frac{p-1}{2^p}\frac{t^{1-p}}{1-t} \quad ^\forall  t\in(0,\tfrac12].
\end{equation}
If $\frac{u(y)+d}{u(x)+d}\in(0,\tfrac12]$, by the second inequality in \eqref{logpf0},
\[\begin{aligned}
\frac{1-(\frac{u(y)+d}{u(x)+d})^{1-p_2}}{1-\frac{u(y)+d}{u(x)+d}} & \le  - \frac{p_2-1}{2^{p_2}} \frac{(\frac{u(y)+d}{u(x)+d})^{1-p_2}}{1-\frac{u(y)+d}{u(x)+d}} = -\frac{(p_2-1)(u(x)+d)^{p_2}}{2^{p_2}(u(x)-u(y))(u(y)+d)^{p_2-1}} \\
&\le  -\frac{(p_2-1)}{2^{p_2}}\left(\frac{u(x)+d}{u(y)+d}\right)^{p_2-1}\le -\frac{p^--1}{2}, 
\end{aligned}\]
hence, using this inequality and \eqref{Rselect5} and choosing $\delta\le \frac{p^--1}{4}$ ,
\[\begin{split}
J & \le  -\frac{p^--1}{4}\frac{1}{(u(x)-u(y))^{p_2-p(x,y)}}\left(\frac{u(x)-u(y)}{u(y)+d}\right)^{p_2}\\
&\le -\frac{c}{(1+\|u\|_{L^\infty(\Omega')})^{\omega_p(r)}} \left(\frac{u(x)+d}{2(u(y)+d)}\right)^{p_2}\eta(y)^{p_2} K(x,y)\\
&\le   -c\left|\ln\left(\frac{u(x)+d}{u(y)+d}\right)\right|^{p_2}\min\{\eta(x),\eta(y)\}^{p_2} K(x,y).
\end{split}\]
On the other hand, if $\frac{u(y)+d}{u(x)+d}\in[\tfrac12,1)$, 
\[
\frac{u(x)-u(y)}{u(x)+d} \ge \frac{1}{2} \frac{u(x)-u(y)}{u(y)+d}  \ge \frac{1}{2} \ln \left(1+\frac{u(x)-u(y)}{u(y)+d}\right) = \frac{1}{2} \ln \left(\frac{u(x)+d}{u(y)+d}\right),
\]
hence, using this inequality, the first inequality in \eqref{logpf0} and $(u(x)-u(y))^{p_2-p(x,y)}\le (1+2\|u\|_{L^\infty(\Omega')})^{\omega_p(r)}\le c$ by  \eqref{Rselect5}, and choosing $\delta\le \frac{p^--1}{2}$, 
\[
J \le - c \left|\ln \left(\frac{u(x)+d}{u(y)+d}\right)\right|^{p_2} \min\{\eta(x),\eta(y)\}^{p_2} K(x,y).
\]
By the symmetry of $p(\cdot,\cdot)$ and $K(\cdot,\cdot)$, we have the same estimate for $J$ when $u(y)> u(x)$. Therefore, we have 
\begin{equation}\label{logpf2}\begin{split}
I_1 & \le -c\int_{B_{2\rho}}\int_{B_{2\rho}} \left|\ln \left(\frac{u(x)+d}{u(y)+d}\right)\right|^{p_2} \min\{\eta(x),\eta(y)\}^{p_2} K(x,y) \,dy\,dx  \\
&\qquad  +c \int_{B_{2\rho}}\int_{B_{2\rho}} |\eta(x)^\frac{p_2}{p(x,y)}-\eta(y)^\frac{p_2}{p(x,y)}|^{p(x,y)}K(x,y) \,dy\,dx.
\end{split}\end{equation}
We further estimate the second integral on the right hand side. Since $|\eta(x)^\frac{p_2}{p(x,y)}-\eta(y)^\frac{p_2}{p(x,y)}|\le c \rho^{-1}|x-y|$ for every $x,y\in B_{2\rho}$,  
\begin{equation}\label{logpf3}\begin{split}
&\int_{B_{2\rho}}\int_{B_{2\rho}}   |\eta(x)^\frac{p_2}{p(x,y)}-\eta(y)^\frac{p_2}{p(x,y)}|^{p(x,y)} K(x,y)\,dy\,dx\\
&  \le c \rho^{-p_2} \int_{B_{2\rho}}\int_{B_{2\rho}} |x-y|^{-n+(1-s(x,y))p(x,y)}\,dy\,dx \\
&  \le c \rho^{-p_2} \int_{B_{2\rho}}\int_{B_{4\rho}(x)} |x-y|^{-n+(1-s_2)p_1}\,dy\,dx \\
& \le c \rho^{n-s_2p_1+p_1-p_2} . 
\end{split}\end{equation}

We next estimate 
\[
I_2 \le \int_{B_{2\rho}} \int_{\R^n\setminus B_{2\rho}}\frac{ (u(x)-u(y))_+^{p(x,y)-1}  \eta(x)^{p_2}}{(u(x)+d)^{p_2-1}} K(x,y)\,dy\,dx.
\]
Since $u\ge 0$ in $B_r$, for  $x\in B_{2\rho}$  and $y\in B_r$
\[\begin{split}
&\frac{ (u(x)-u(y))_+^{p(x,y)-1} }{(u(x)+d)^{p_2-1}}  \le (u(x)+d)^{p(x,y)-p_2}\\
& \le 
\begin{cases}
[(1+\|u\|_{L^\infty(\Omega')})\max\{1,d\}]^{p_4-p_3} \overset{\eqref{Rselect5}}{\le} c \max\{1,d\}^{p_4-p_3}\ \ \text{if}\ \ p(x,y)>p_2 \\
\max\{1,d^{-1}\}^{p_4-p_3}\ \ \text{if}\ \ p(x,y)\le p_2
\end{cases}\\
& \le c \max\{d,d^{-1}\}^{p_4-p_3}
\end{split}\]
and for $x\in B_{2\rho}$ and $y\in \R^n\setminus B_r$
\[
\frac{ (u(x)-u(y))_+^{p(x,y)-1} }{(u(x)+d)^{p_2-1}} \le  \frac{ (\|u\|_{L^\infty(B_r)}+(u(y))_-)^{p(x,y)-1}}{d^{p_2-1}}.
\]
Therefore,
\[ \begin{split}
&I_2  \le c  \max\{d,d^{-1}\}^{p_4-p_3} \int_{B_{2\rho}} \int_{B_r\setminus B_{2\rho}} \frac{\eta(x)^{p_2}}{|x-y|^{n+s(x,y)p(x,y)}} \,dy\,dx \\
&\qquad + c\int_{B_{2\rho}} \int_{\R^n\setminus B_{r}}\frac{ [\|u\|_{L^\infty(B_r)}+(u(y))_-]^{p(x,y)-1}}{d^{p_2-1}|x-y|^{n+s(x,y)p(x,y)}} \,dy\,dx\\
& \le  c\max\{d,d^{-1}\}^{p_4-p_3} \int_{B_{\frac{3\rho}{2}}} \int_{B_{r}\setminus B_{\frac\rho2}(x)} \frac{1}{|x-y|^{n+s_4p_4}} \,dy\,dx+ \frac{c \rho^n}{d^{p_2-1}} T (u_-+\|u\|_{L^\infty(B_r)},r,2\rho) \\
&\le   c \max\{d,d^{-1}\}^{p_4-p_3}  \rho^{n-s_4p_4} + \frac{c \rho^n}{d^{p_2-1}}T (u_-+\|u\|_{L^\infty(B_r)},r,2\rho).
\end{split}\]

Finally, combining \eqref{logpf2}, \eqref{logpf3} and the preceding inequality with $I_1+2I_2=0$, we get the desired estimate.
\end{proof}

The preceding lemma directly implies the next result.
\begin{lemma}\label{corlog}
Let $u\in\mathbb W^{s(\cdot,\cdot),p(\cdot,\cdot)}(\Omega)\cap L^\infty(\Omega;L^{p(\cdot,\cdot)-1}_{s(\cdot,\cdot)p(\cdot,\cdot)}(\R^n))$ be a weak  solution to \eqref{mainPDE}
such that $u\in L^\infty(\Omega')$ for some $\Omega'\Subset\Omega$. If $r\in(0,1)$ satisfies
\eqref{Rselect5}
and $u\ge 0$ in $B_r=B_r(x_0)\subset \Omega'$,
then for any $B_\rho=B_\rho(x_0)$ with $0<\rho< r/2$, we have
\[\begin{split}
\fint_{B_\rho}|v-(v)_{B_\rho}|^{p_2}\,dx & \le
  c\rho^{(s_1-s_2)p_1+p_1-p_2} +c\max\{d,d^{-1}\}^{p_4-p_3} \rho^{s_1p_1-s_4p_4} \\
 &\qquad +  \frac{c \rho^{s_1p_1}}{d^{p_2-1}}T (u_-+\|u\|_{L^\infty(B_r)},r,2\rho)
\end{split}\]
for some $c=c(n,s^\pm,p^\pm,\Lambda)>0$, where
\[
v:= \min\{(\ln(a+d)-\ln(u+d))_+,\ln b\} \quad \text{with }\ a,d>0\ \ \text{and}\ \  b>1,
\]
and $s_i$ and $p_i$ ($i=1,2,3,4$) are given in \eqref{si} and \eqref{pi}.
\end{lemma}
\begin{proof}
We first notice that
\[
|v(x)-v(y)|\le |\ln (u(x)+d)-\ln(u(y)+d)|=\left|\ln\left(\frac{u(x)+d}{u(y)+d}\right)\right|.
\]
%Since $\frac{np_2}{n+\frac{s_1}{4}p_2} < p_1$ by \eqref{Rselect5}, applying  Sobolev-Poincar\'e's inequality for fractional Sobolev space in Lemma~\ref{lemSPineq} and Lemma~\ref{lempxpq}, 
From this inequality and H\"older's inequality, we have
\[\begin{split}
\fint_{B_\rho}|v-(v)_{B_\rho}|^{p_2}\,dx & \le \fint_{B_\rho} \fint_{B_\rho}|v(x)-v(y)|^{p_2}\,dy\,dx \\
&\le  c \rho^{s_1p_1-n} \int_{B_\rho}\int_{B_\rho}\frac{|v(x)-v(y)|^{p_2}}{|x-y|^{n+s(x,y)p(x,y)}} \,dy\,dx\\
& \le c \rho^{s_1p_1-n} \int_{B_\rho}\int_{B_\rho}\left|\ln\left(\frac{u(x)+d}{u(y)+d}\right)\right|^{p_2}K(x,y) \,dy\,dx.
\end{split}\]
Therefore, by Lemma~\ref{lemlog}, we have the conclusion.
\end{proof}

\begin{remark}
The estimate in Lemma~\ref{lemcaccio} with the ``$+$"(resp. ``-") sign still holds for subsolutions(resp. supersolutions) $u\in \mathbb W^{s(\cdot,\cdot),p(\cdot,\cdot)}(\Omega)\cap L^\infty(\Omega;L^{p(\cdot,\cdot)-1}_{s(\cdot,\cdot)p(\cdot,\cdot)}(\R^n))$ to \eqref{weakform}. On the other hand, the estimates in Lemmas~\ref{lemlog} and \ref{corlog} with the ``$+$"(resp. ``-") sign still hold for supersolutions $u\in \mathbb W^{s(\cdot,\cdot),p(\cdot,\cdot)}(\Omega)\cap L^\infty(\Omega;L^{p(\cdot,\cdot)-1}_{s(\cdot,\cdot)p(\cdot,\cdot)}(\R^n))$ to \eqref{weakform}.
\end{remark}

\section{\bf Local boundedness}\label{sec5}

In this section, we prove the local boundedness of the weak solutions to \eqref{mainPDE}.

\begin{proof}[{\bf Proof of Theorem~\ref{thm1}}]
Since $\lim_{r\to0}\omega(r)=0$, one can find $r>0$ satisfying that \eqref{Rselect2} and 
\begin{equation}\label{Rselect3}
\omega_{p}(r)\le \min\left\{\frac{s^-}{4},\frac{3s^-(p^-)^2}{4n},\frac{p^-}{4}\right\}.
\end{equation}
Fix any $B_{2r}=B_{2r}(x_0)\Subset \Omega$, and set
\[
s_1:=s^-_{B_{2r}}, \quad s_2:=s^+_{B_{2r}}, \quad  p_1:=p^-_{B_{2r}}, \quad  p_2:=p^+_{B_{2r}}.
\]
Note that if $s_1p_1>n$ then   $u\in W^{\tilde s, p_1}(B_r)\subset L^\infty(B_r)$ by Lemma~\ref{lempxpq} and \cite[Theorem 8.2]{DPV1}, where  $\tilde s\in (0,s)$ is arbitrary satisfying $\tilde s p_1>n$. Therefore, we assume that $s_1 p_1\le n$.
 
We only prove $\underset{B_{r/2}}{\mathrm{ess\, sup}}\, u_+<\infty$ . Then, since $-u$ is also a weak solution to \eqref{weakform}, we also obtain $\underset{B_{r/2}}{\mathrm{ess\, sup}}\, u_-<\infty$.

Define
\[
\sigma := \max\{2s^--1,\tfrac{3}{4}s^-\} \ \ \Big(\Longleftrightarrow \ \ s^--\sigma = \min\{1-s^-, \tfrac{1}{4}s^-\}\Big).
\]
Then we immediately see that  $\tfrac{3}{4}s^- \le \sigma < s^-$.  Moreover, by \eqref{Rselect3}, we have $p_2-p_1\le \frac{3s^-p_2p_1}{4n}$ hence $p_2< (p_1)_{\frac{3}{4}s^-}^*\le (p_1)_{\sigma}^*$. Note that $n-\sigma p_1 >n-s^-p_1\ge 0$ hence $(p_1)_{\sigma}^* := \frac{np_1}{n-\sigma p_1}$ is well defined.

Let $r/2\le r_1<r_2 \le r$, and $w_k:=(u-k)_+$ with $k\ge 0$. Then by H\"older's inequality and Lemma~\ref{lemSPineq}, we have 
\[\begin{split}
& \fint_{B_{r_1}}w_k^{p_2} \,dx   \le \bigg(\frac{|A^+(k,r_1)|}{|B_r|}\bigg)^{1-\frac{p_2}{(p_1)_{\sigma}^*}} \bigg(\fint_{B_{r_1}}w_k^{(p_1)_{\sigma}^*} \,dx\bigg)^{\frac{p_2}{(p_1)_{\sigma}^*}}\\
& \le c\bigg(\frac{|A^+(k,r_1)|}{|B_r|}\bigg)^{1-\frac{p_2}{(p_1)_{\sigma}^*}}  \bigg[r^{\sigma p_1}\fint_{B_{r_1}}\int_{B_{r_1}}\frac{|w_k(x)-w_k(y)|^{p_1}}{|x-y|^{n+\sigma p_1}}\,dy\,dx +\fint_{B_{r_1}}w_k^{p_1}\,dx \bigg]^{\frac{p_2}{p_1}}.
\end{split}\]
We first estimate the double integral on the right hand side. Applying  Lemma~\ref{lempxpq} to $v=w_
k$, $\Omega=B_{r_1}$, $s_1(x,y)\equiv \sigma$, $s_2(x,y)\equiv s(x,y)$, $p_1(x,y)\equiv p_1$, $p_2(x,y)\equiv p(x,y)$, and Lemma~\ref{lemcaccio} to $\rho=r_1$ and $r=r_2$,
\[\begin{split}
&\varrho_{\sigma,p_1}(w_k,B_{r_1}) \le  c \varrho_{s(\cdot,\cdot),p(\cdot,\cdot)}(w_k,B_{r_1})+c|A(k,r_1)|\\
&\ \le c \frac{r_2^{(1-s_2)p_2}}{(r_2-r_1)^{p_2}}\bigg(\int_{B_{r_2}} w_k^{p_2}\,dx+ |A(k,r_2)|\bigg)+ c \frac{r_2^{n+s_2p_2}}{(r_2-r_1)^{n+s_2p_2}} \bigg(  \int_{B_{r_2}}w_k \,dx \bigg) T(w_k,r_1,r_1).
\end{split}\]
Note that when applying Lemma~\ref{lempxpq} we use the facts that $r_1\le 1$ and $d_1\ge s^--\sigma$.
Therefore, from the preceding two estimates and using the facts that  $\frac{r^{(1-s_2)p_2+\sigma p_1}}{(r_2-r_1)^{p_2}}\ge 1$, $w_k\le u_+$ and $r/2\le r_1 \le r$, we have 
\[%begin{equation}\label{Degiorgi0}
\begin{split}
 \fint_{B_{r_1}}w_k^{p_2} \,dx  & \le c\bigg(\frac{|A^+(k,r_2)|}{|B_r|}\bigg)^{1-\frac{p_2}{(p_1)_{\sigma}^*}} \bigg[\frac{r^{(1-s_2)p_2+\sigma p_1}}{(r_2-r_1)^{p_2}}\bigg(\fint_{B_{r_2}} w_k^{p_2} \,dx+\frac{|A(k,r_2)|}{|B_r|}\bigg) \\
&\qquad  \qquad \qquad \qquad \qquad \qquad+  \frac{r^{n+s_2p_2+\sigma p_1}}{(r_2-r_1)^{n+s_2p_2}} \bigg(  \fint_{B_{r_2}}w_k \,dx \bigg) T(u_+,\tfrac r2,r)  \bigg]^{\frac{p_2}{p_1}}.
\end{split}
\]%end{equation}

Let $0< h<k$. Then we notice that for any $\rho\in[\frac{r}{2},r]$,
\[
\frac{|A^+(k,\rho)|}{|B_r|} \le \frac{1}{|B_r|} \int_{A^+(k,\rho)} \frac{w_h^{p_2}}{(k-h)^{p_2}}\,dx \le  \frac{1}{(k-h)^{p_2}} \fint_{B_\rho} w_h^{p_2} \,dx,
\]
\[
\fint_{B_\rho} w_k^{p_2} \,dx \le \fint_{B_\rho} w_h^{p_2} \,dx 
\qquad \text{and}\qquad
\fint_{B_\rho} w_k \,dx \le \frac{1}{(k-h)^{p_2-1}} \fint_{B_\rho} w_h^{p_2} \,dx.
\]
Hence,
\begin{equation}\label{Degiorgi1}\begin{split}
 &\fint_{B_{r_1}}w_k^{p_2} \,dx   \le  \frac{c}{(k-h)^{p_2(1-\frac{p_2}{(p_1)_{\sigma}^*})}} \bigg[\frac{r^{(1-s_2)p_2+\sigma p_1}}{(r_2-r_1)^{p_2}}\left(1+\frac{1}{(k-h)^{p_2}} \right)\\
&\qquad\qquad\qquad\qquad\qquad\qquad +   \frac{r^{n+s_2p_2+\sigma p_1}}{(r_2-r_1)^{n+s_2p_2}}  \frac{T(u_+,\frac r2,r)}{(k-h)^{p_2-1}}  \bigg]^{\frac{p_2}{p_1}} \bigg(\fint_{B_{r_2}} w_h^{p_2} \,dx\bigg)^{1+\sigma_0},
\end{split}\end{equation}
where 
\[
\sigma_0 :=\frac{p_2}{p_1}-\frac{p_2}{(p_1)_{\sigma}^*} =
\frac{\sigma p_2}{n}.
\]

Now we define for $i=0,1,2,\dots$,
\[
k_i :=  M(1-2^{-i})  \ \text{ with }\   M>0,
\quad
\rho_i := (1+2^{-i})\frac{r}{2}, 
\quad 
Y_i:= \fint_{B_{\rho_i}} w_{k_i}^{p_2}\,dx.
\]
By \eqref{Degiorgi1} with $k=k_{i+1}$, $h=k_i$, $r_1=\rho_i$ and $r_2=\rho_{i+1}$, we have
\[\begin{split}
 Y_{i+1}  & \le c \left(\frac{2^i}{M}\right)^{p_2(1-\frac{p_1}{p_2}+\sigma_0)} \bigg[2^{p_2i}r^{\sigma p_1-s_2p_2}\left(1+\frac{2^{p_2i}}{M^{p_2}}\right)\\
&\qquad\qquad\qquad\qquad\qquad + 2^{(n+s_2p_2+p_2-1)i}  \frac{ r^{\sigma p_1} T(u_+,\frac{r}{2},r)}{M^{p_2-1}}\bigg]^{\frac{p_2}{p_1}} Y_i^{1+\sigma_0}.
\end{split}\]
Fixed an arbitrary $\delta\in(0,1]$. We choose  $M>0$ such that
\[
M \geq  M_1:= \delta [ r^{s_2p_2} T(u_+,\tfrac r2,r)]^{\frac{1}{p_2-1}} + \delta^{\frac{p_2-1}{p_2}}.
\] 
Then we have
\[
Y_{i+1} \le c_0 \delta^{-p_2+1}  r^{\sigma p_1-s_2p_2} M^{-p_2(1-\frac{p_2}{p_1}+\sigma_0)} 2^{\alpha i} Y_i^{1+\sigma_0}
\]
for some $c_0,\alpha>0$ depending only on $n$, $s^\pm$, $p^\pm$ and $\Lambda$. Furthermore, if we choose $M>0$ such that
\[
M  \ge M_2:=\left[ c_0^{\frac{1}{\sigma_0}}  \delta^{-\frac{p_2-1}{\sigma_0}} r^{\frac{\sigma p_1-s_2p_2}{\sigma_0}}2^{\frac{\alpha}{\sigma_0^2}}\left(\fint_{B_{r}} u_+^{p_2}\,dx\right)\right]^{\frac{\sigma_0}{p_2(1-\frac{p_2}{p_1}+\sigma_0)}},
\]  
then
\[
Y_0 =\fint_{B_{\frac r2}} u_+^{p_2}\,dx \le  c_0^{-\frac{1}{\sigma_0}}  \delta^{\frac{p_2-1}{\sigma_0}} r^{-\frac{\sigma p_1-s_2p_2}{\sigma_0}}M^{\frac{p_2(1-\frac{p_2}{p_1}+\sigma_0)}{\sigma_0}} 2^{-\frac{\alpha}{\sigma_0^2}}.
\]  
Therefore, by Lemma~\ref{lemseq} we get $\lim_{i\to \infty}Y_i=0$ which implies
\[
\underset{B_{\frac r2}}{\mathrm{ess\, sup}}\, u_+ \le M,
\]
where $M$ can be chosen as  $M=M_1+M_2$ so that
\[\begin{split}
%M & =M_1+M_2\\
& M \le c \left[\delta^{-\frac{p_2-1}{\sigma_0}} r^{\frac{\sigma p_1-s_2p_2}{\sigma_0}}\left(\fint_{B_{r}} u_+^{p_2}\,dx\right)\right]^{\frac{\sigma_0}{p_2(1-\frac{p_2}{p_1}+\sigma_0)}} + \delta [ r^{s_2p_2} T(u_+,\tfrac{r}{2},r)]^{\frac{1}{p_2-1}} + \delta^{\frac{p_2-1}{p_2}}
\end{split}\]
for some  $c=c(n,s^\pm,p^\pm,\Lambda)>0$.
\end{proof}

\begin{remark}
In the above theorem, we can also obtain the following $L^\infty$-estimate: 
\[\begin{split}
 \|u\|_{L^\infty(B_{r/2})} &\le  c \left[\delta^{-\frac{p_2-1}{\sigma_0}} r^{\frac{\sigma p_1-s_2p_2}{\sigma_0}}\left(\fint_{B_{r}} |u|^{p_2}\,dx\right)\right]^{\frac{\sigma_0}{p_2(1-\frac{p_2}{p_1}+\sigma_0)}}\\
&\qquad\qquad\qquad\qquad + \delta [ r^{s_2p_2} T(|u|,\tfrac{r}{2},r)]^{\frac{1}{p_2-1}} + \delta^{\frac{p_2-1}{p_2}},
\end{split}\]
where $\delta\in(0,1]$ is  arbitrary. (Note that we can also have almost the same estimate in the case  $s_1p_1>n$ by using the same argument in the above proof.)
It looks quite complicated. If we assume that $p(\cdot,\cdot)$ is log-H\"older continuous on the diagonal region $\{(x,x): x\in\Omega\}$, that is, it satisfies \eqref{logholder} without $\omega_s(r)$, then for every $B_{2r}\Subset\Omega$ with $r>0$ satisfying  \eqref{Rselect2}, \eqref{Rselect3} and $\omega_{p}(r)\le \frac{\ln 2}{\ln (\varrho_{s(\cdot,\cdot),p(\cdot,\cdot)}(u,\Omega)+\delta+1)}$ we have
\[
 \|u\|_{L^\infty(B_{r/2})} \le  c \delta^{-\frac{p_2-1}{\sigma_0p_2}} r^{\frac{\sigma-s_2}{\sigma_0}}\left(\fint_{B_{r}} |u|^{p_2}\,dx\right)^{\frac{1}{p_2}} + \delta [ r^{s_2p_2} T(|u|,\tfrac{r}{2},r)]^{\frac{1}{p_2-1}} + \delta^{\frac{p_2-1}{p_2}}.
\]
Moreover, if $s(\cdot,\cdot)\equiv s$ and $p(\cdot,\cdot)\equiv p$ with $sp<n$, then we can choose $\sigma=s$ in the preceding proof and obtain the last estimate without  the term $\delta^{\frac{p_2-1}{p_2}}$, that is, we have 
\[ 
 \|u\|_{L^\infty(B_{r/2})} \le  c \delta^{-\frac{(p-1)n}{sp^2}} \left(\fint_{B_{r}} |u|^{p}\,dx\right)^{\frac{1}{p}} + \delta [ r^{sp} T(|u|,\tfrac{r}{2},r)]^{\frac{1}{p-1}}.
\]
This is exactly the same as the $L^\infty$-estimates for the nonlocal equations with constant powers obtained in \cite[Theorem 1.1]{DKP2} (Note that the definition of the nonlocal tail in \cite{DKP2} is slightly different from the one in this paper).
\end{remark}

\begin{remark}\label{rmksmall}
Instead of the continuity assumption of $p(\cdot,\cdot)$ on the diagonal region in Theorem~\ref{thm1}, we can assume that the oscillation of $p(\cdot,\cdot)$, $\omega_p(r)$, is sufficiently small. Precisely, we may assume that $\omega_p(\cdot)$ satisfies \eqref{Rselect2} and  \eqref{Rselect3} for some $r<\frac{1}{2}$.
\end{remark}

\section{\bf H\"older continuity}\label{sec6}

Finally, we prove the local H\"older continuity of weak solutions to \eqref{mainPDE} by proving  Theorem~\ref{mainthm}. Therefore, we assume that $s(\cdot,\cdot)$ and $p(\cdot,\cdot)$ satisfy \eqref{logholder} in this section. 

Fix any $\Omega'\Subset \Omega$. Note that since we have seen $u\in L^\infty_{\mathrm{loc}}(\Omega)$ in the preceding section, $\|u\|_{L^\infty(\Omega')}<\infty$. Let $\sigma=\sigma(n,\Lambda,s^\pm,p^\pm,c_{LH})\in(0,\tfrac14)$ be a  small positive number that will be determined later in \eqref{sigma1} and \eqref{sigma2}.  Suppose that $r>0$ satisfies  \eqref{Rselect2}, \eqref{Rselect5}, \eqref{Rselect3},
\begin{equation}\label{Rselect6}
r\le  \sigma^{\frac{s^+p^+}{s^-(p^{+}-1)}-1},
\end{equation}
\[%begin{equation}\label{Rselect4'}
\omega_{s}(r)\le \frac{\ln2}{\ln(1/\sigma)},
\]%end{equation}

\begin{equation}\label{Rselect4}
\omega_{p}(r) \le \min\bigg\{\frac{s^-p^-}{2s^+},(1-s^+)p^-,\frac{\ln2 }{\ln ([v]_{L^\infty(\Omega';L^{p(\cdot,\cdot)-1}_{s(\cdot,\cdot)p(\cdot,\cdot)}(\R^n))}+\|u\|_{L^\infty(\Omega')}+1/\sigma+R_{\Omega'})}\bigg\},
\end{equation}
where $R_{\Omega'}:=\sup\{|x|: x\in \Omega'\}<\infty$.
Note that the above inequalities yield
\begin{equation}\label{sigmap2p1}
\sigma^{-\omega_{s}(r)} \le 2 \quad \text{and}\quad
([v]_{L^\infty(\Omega';L^{p(\cdot,\cdot)-1}_{s(\cdot,\cdot)p(\cdot,\cdot)}(\R^n))}+\|u\|_{L^\infty(\Omega')}+R_{\Omega'}+1/\sigma)^{\omega_p(r)}\le 2 .
\end{equation}

Let $B_r=B_r(x_0)\subset \Omega'$, and  define
\begin{equation}\label{K0}
K_0 :=2 \|u\|_{L^\infty(B_r)}+\left[r^{s_0p_0}T(|u|+\|u\|_{L^\infty(B_r)},x_0,r, r)\right]^{\frac{1}{p^--1}} +1,
\end{equation}
where 
\[
s_0:=s(x_0,x_0)\quad \text{and}\quad p_0:=p(x_0,x_0).
\]
For $j=0,1,2,\dots$, set 
\[
r_j:= \sigma^j\frac{r}{2}, \quad  B_j:= B_{r_j}(x_0), \quad  2B_j:= B_{2r_j}(x_0),
\]
\[
s_{j,1}:=s^-_{2B_j} \quad   s_{j,2}:=s^+_{2B_j}, \quad  p_{j,1}:=p^-_{2B_j} \quad \text{and}\quad  p_{j,2}:=p^+_{2B_j}.
\]
In particular, we write 
\[
 p_1:=p_{0,1}\quad \text{and}\quad p_2:=p_{0,2}.
\]
Note that by \eqref{logholder}
\begin{equation}\label{logholderj}
r_j^{-(s_{j,2}-s_{j,1})}\le r_j^{-\omega_{s}(2r_j)} \le c \quad \text{and} \quad  r_j^{-(p_{j,2}-p_{j,1})}\le r_j^{-\omega_{p}(2r_j)} \le c
\end{equation}
and that by \eqref{T1} and \eqref{sigmap2p1}
\[\begin{split}
T(|u|+\|u\|_{L^\infty(B_r)},x_0,r, r) & \le \left(1+\frac{R_{\Omega'}+1}{r}\right)^{n+s^+p^{+}}[v+\|u\|_{L^\infty(\Omega')}]_{L^\infty(\Omega';L^{p(\cdot,\cdot)-1}_{s(\cdot,\cdot)p(\cdot,\cdot)}(\R^n))}\\
& \le c \left(\frac{R_{\Omega'}+1}{r}\right)^{n+s^+p^{+}}\left([v]_{L^\infty(\Omega';L^{p(\cdot,\cdot)-1}_{s(\cdot,\cdot)p(\cdot,\cdot)}(\R^n))}+ \|u\|_{L^\infty(\Omega')}^{p^+-1}+1\right)\\
\end{split}\]
and so
\begin{equation}\label{K0p2p1}
K_0^{p_2-p_1}\le K_0^{\omega_{p}(r)}  \le    c.
\end{equation}

Now we prove an oscillation decay  estimate.

\begin{lemma}\label{lemholder} Under the above setting, we further suppose that 
\begin{equation}\label{holderbasicass}
s_1p_1<n + \frac{s^-p^-}{4}.
\end{equation} 
Then we have
\begin{equation}\label{Holderinduction}
\theta(r_j):= \sup_{B_j} u-  \inf_{B_j} u  \le K_j:=\sigma^{\alpha j} K_0 \qquad \text{for all }\ j=0,1,2,\dots,
\end{equation}
for some $\alpha=\alpha(n,\Lambda,s^\pm,p^\pm, c_{LH})>0$ satisfying that
\begin{equation}\label{alphaassumption}
\alpha \le \min\left\{\frac{s^-p^{-}}{2(p^{+}-1)},\, \ln_{\sigma}\tfrac12,\, \ln_\sigma \Big(1-\sigma^{\frac{s^+p^{+}}{p^{+}-1}}\Big),s^-\right\} .
\end{equation}
\end{lemma}

\begin{proof}

\textit{Step 1 (Induction).}
We first observe from \eqref{logholderj} and \eqref{K0p2p1} that
\begin{equation}\label{kjp2p1}
\sup_{x_1x_2,y_1,y_2\in B_j} K_j^{p(x_1,y_1)-p(x_2,y_2)} \le \sigma^{-\alpha j \omega_{p}(r_j)}K_0^{\omega_{p}(r_j)} \le  r_j^{-\omega_{p}(r_j)} K_0^{\omega_{p}(r)}\le c,
\end{equation}
 from \eqref{alphaassumption} that
\begin{equation}\label{alpha1}
  s^-p^{-}-\alpha (p^{+}-1) \ge  \frac{s^-p^{-}}{2},
\end{equation}
\begin{equation}\label{alpha2}
 \sigma^\alpha \ge \frac{1}{2},
\end{equation}
\begin{equation}\label{alpha3}
 \sigma^{\alpha}-1 +\sigma^{\frac{s(x,y)p(x,y)}{p^+-1}-\alpha} > \sigma^{\alpha}-1 +\sigma^{\frac{s^+p^{+}}{p^{+}-1}} \ge 0 , \quad \text{for all }\ x,y\in \R^n.
\end{equation}

We prove the lemma by induction. Since
\[
\sup_{B_0} u-  \inf_{B_0} u  \le 2\|u\|_{L^\infty(B_0)} \le K_0,
\]
 \eqref{Holderinduction} is true when $j=0$. Now assume that 
\begin{equation}\label{inductioni}
\theta(r_i) \le K_i\qquad \text{for all }\ i=0,1,2,\dots,j,
\end{equation}
and then we will prove that 
\begin{equation}\label{inductionj+1}
\theta(r_{j+1}) \le K_{j+1}.
\end{equation}
Hence, from now on, $j\in\mathbb N$ is fixed. Without loss of generality, we shall assume that
\[
\theta(r_{j+1}) \ge \frac{1}{2}K_{j+1}.
\] 
Then this and \eqref{alpha2} imply that
\begin{equation}\label{thetaj}
\theta(r_j)\ge  \theta(r_{j+1}) \ge \frac{1}{2} K_{j+1} = \frac{1}{2}\sigma^\alpha K_{j} \ge \frac{1}{4} K_j.
\end{equation}
We notice that one of the following two cases must holds:
\[
 \frac{|2B_{j+1}\cap\{u \ge \inf_{B_j}u + \theta(r_j)/2\}|}{|2B_{j+1}|} \ge \frac{1}{2},  
\]
\[
 \frac{|2B_{j+1}\cap\{u \le \inf_{B_j}u + \theta(r_j)/2\}|}{|2B_{j+1}|} \ge \frac{1}{2}.
\]
If the first case is true we define $u_j:=u-\inf_{B_j}u$. On the other hand, if the second case is true we define $u_j:=\sup_{B_j}u- u=\theta(r_j)-(u-\inf_{B_j}u)$. Note that in every case,  $u_j\ge 0$ in $B_j$, $u_j$ is also a weak solution to \eqref{mainPDE}, 
\begin{equation}\label{density}
 \frac{|2B_{j+1}\cap\{u_j \ge  \theta(r_j)/2\}|}{|2B_{j+1}|} \ge \frac{1}{2},  
\end{equation}
and 
\begin{equation}\label{supuj}
\sup_{B_i} u_j \le \theta(r_i) \le K_i\quad \text{for all }\ i=0,1,\dots,j. 
\end{equation}

Finally, we define 
\begin{equation}\label{depsilon}
d_j:= \epsilon \theta(r_j)\quad \text{with }\ \epsilon:=\sigma^{\frac{s_0p_0}{p^+-1}-\alpha}\overset{\eqref{alpha1}}{\le}  \sigma^{ \frac{s^-p^-}{2(p^+-1)}}<1.
\end{equation}
Note that by  \eqref{inductioni} with $i=j$, \eqref{thetaj},  \eqref{kjp2p1} and \eqref{sigmap2p1}
\begin{equation}\label{dp2p1}\begin{split}
 \max\{d_j,d_j^{-1}\}^{p_{j,2}-p_{j,1}}  \le  c \epsilon^{-({p_{j,2}-p_{j,1}})}\max\{K_j,K_j^{-1}\}^{p_{j,2}-p_{j,1}}  \le   c \sigma^{-\omega_p(r)\frac{s^+p^{+}}{p^{+}-1}}  \le c.
\end{split}\end{equation}

\textit{Step 2 (Tail estimates).}
We first estimate  $\frac{r_{j+1}^{s_0p_0}}{d_j^{p_0-1}}T(|u_j|+\|u_j\|_{L^\infty(B_j)},r_j,r_j)$. Here, we note that  $\|u_j\|_{L^\infty(B_j)}=\theta(r_j)$. From  
\eqref{supuj}, \eqref{inductioni}, \eqref{depsilon}, \eqref{thetaj} and \eqref{kjp2p1}  with the  the definitions of $u_j$, the tail $T$ and $K_0$ in \eqref{K0}, we have that for every $x\in B_j$,
\[\begin{aligned}
& d_j^{-(p_0-1)} T(|u_j|+\|u_j\|_{L^\infty(B_r)},r_j,r_j)  \\
 & =\frac{1}{d_j^{p_0-1}} \left[\sum_{i=1}^j \int_{B_{i-1}\setminus B_i} \frac{(|u_j(y)|+\theta(r_j))^{p(x,y)-1}}{|y-x_0|^{n+s(x,y)p(x,y)}}\,dy + \int_{\R^n\setminus B_{0}} \frac{(|u_j(y)|+\theta(r_j))^{p(x,y)-1}}{|y-x_0|^{n+s(x,y)p(x,y)}}\,dy\right]\\
 & \le \frac{c}{(\epsilon K_j)^{p_0-1}} \left[\sum_{i=1}^j \int_{B_{i-1}\setminus B_i} \frac{K_{i-1}^{p_0-1+p(x,y)-p_0}}{|y-x_0|^{n+s(x,y)p(x,y)}}\,dy +  \int_{\R^n\setminus B_{0}} \frac{(|u|+\|u\|_{L^\infty(B_r)})^{p(x,y)-1}}{|y-x_0|^{n+s(x,y)p(x,y)}}\,dy\right]\\
 &\le c \epsilon^{-(p_0-1)}\bigg[ \sum_{i=1}^j \sigma^{\alpha(p_0-1)(i-1-j)} \int_{B_{i-1}\setminus B_i} \frac{K_{i-1}^{p(x,y)-p_0}}{|y-x_0|^{n+s_{i-1,2}p_{i-1,2}}}\,dy  \\
 &\qquad\qquad\qquad\qquad\qquad\qquad\qquad + \frac{\sigma^{-\alpha j(p_0-1)} }{K_0^{p^--1}} T(|u|+\|u\|_{L^\infty(B_r)},r,r)\bigg]\\
&\le c\epsilon^{-(p^+-1)} \left( \sum_{i=1}^j  r_i^{-s_{i-1,2}p_{i-1,2}} \sigma^{\alpha(p^+-1)(i-1-j)} +r^{-s_0p_0}\sigma^{-\alpha(p^+-1)j}\right)\\
&\le c\epsilon^{-(p^+-1)} r^{-s_0p_0}\sigma^{-\alpha(p^+-1)}  \sum_{i=1}^j  \sigma^{-s_0p_0i} \sigma^{\alpha(p^+-1)(i-j)}.
\end{aligned}\]
In the last inequality we used the inequality 
\[
r_i^{-s_{i-1,2}p_{i-1,2}}= (\sigma r_{i-1})^{-s_{i-1,2}p_{i-1,2}}\le (\sigma r_{i-1})^{-s_0p_0}=(\sigma^i r)^{-s_0p_0},
\]
which follows from \eqref{sigmap2p1} and \eqref{logholderj}. Therefore, recalling the definition of $\epsilon$ in \eqref{depsilon} and using \eqref{alpha1} and the fact that $\sigma\in(0,\frac14)$, we get
\begin{equation}\label{tailj}\begin{aligned}
\frac{r_{j+1}^{s_0p_0}}{d_j^{p_0-1}}T(|u_j|+\|u_j\|_{L^\infty(B_j)},r_j,r_j)&\le c\epsilon^{-(p^+-1)} \sigma^{s_0p_0-\alpha(p^+-1)} \sum_{i=1}^j   \sigma^{(-s_0p_0+\alpha(p^+-1))(i-j)} \\
& \le c \sum_{i=0}^\infty   4^{-\frac{s^-p^-}{2}}\le c.
\end{aligned}\end{equation}

We next estimate $r_j^{s_0p_{0}} T(\theta(r_j),r_j,r_j)$ in a similar way.  By  \eqref{inductioni} with $i=j$,  \eqref{K0}, \eqref{K0p2p1}, and \eqref{kjp2p1}, we have that  for every $x\in B_j$,
\[\begin{aligned}
& \int_{\R^n\setminus B_{j}}\frac{\theta(r_j)^{p(x,y)-1}}{|y-x_0|^{n+s(x,y)p(x,y)}}\,dy \\
 & = \sum_{i=1}^j \int_{B_{i-1}\setminus B_i} \frac{\theta(r_j)^{p(x,y)-1}}{|y-x_0|^{n+s(x,y)p(x,y)}}\,dy +  \int_{\R^n\setminus B_{0}} \frac{\theta(r_j)^{p(x,y)-1}}{|y-x_0|^{n+s(x,y)p(x,y)}}\,dy\\
&\le c  \sum_{i=1}^j \int_{B_{i-1}\setminus B_i} \frac{K_{i-1}^{p(x,y)-1}}{|y-x_0|^{n+s_{i-1,2}p_{i-1,2}}}\,dy + c  \int_{\R^n\setminus B_{0}} \frac{\|u\|_{L^\infty(B_r)}^{p(x,y)-1}}{|y-x_0|^{n+s(x,y)p(x,y)}}\,dy \\
&\le c  \sum_{i=1}^j  K_{i-1}^{p_0-1} r_i^{-s_{i-1,2}p_{i-1,2}} + c r^{-s_0p_0} K_0^{p^--1}  \le c    \sum_{i=1}^j K_{i-1}^{p_{0}-1} r_i^{-s_0p_{0}},
\end{aligned}\]
hence
\begin{equation}\label{jtail}\begin{split}
r_j^{s_0p_{0}} T(\theta(r_j),r_j,r_j) 
&  \le (\sigma^jr_0)^{s_0p_{0}} \sum_{i=1}^j [\sigma^{(i-1)\alpha} K_0]^{p_{0}-1} (\sigma^ir_0)^{-s_0p_{0}}\\
&  = \sigma^{-\alpha(p_{0}-1)} (\sigma^{\alpha j}  K_0)^{p_0-1}\sum_{i=1}^j \sigma^{(s_0p_{0}-\alpha(p_{0}-1))(j-i)}  \\
 &\overset{\eqref{Holderinduction},\eqref{alpha1}}{\le}   \sigma^{-\alpha(p_0-1)} K_j^{p_0-1}  \sum_{i=0}^\infty 4^{-\frac{s^-p^-}{2}i}\\
  &\overset{\eqref{thetaj}}{\le}  c \sigma^{-\alpha(p_0-1)} \theta(r_j)^{p_{0}-1} .
 \end{split}\end{equation}

\textit{Step 3 (Density estimates).}
Define
\begin{equation}\label{vdef}
v:=\min\left\{\left[\ln\left(\frac{\theta(r_j)/2+d_j}{u_j+d_j}\right)\right]_+, k\right\}, 
\end{equation}
where $d_j$ is give in \eqref{depsilon}, and  a constant $k>0$ will be chosen later in \eqref{kdef}.
%We have from  the mean value theorem, \eqref{dp2p1} and \eqref{Rselect4} that  for every $x,y\in 2B_{j+1}$,
%\[\begin{split}
%|v(x)-v(y)|^{p_{j,2}-p_{j,1}} & \le |\ln(u_j(x)+d)-\ln(u_j(y)+d)|^{p_{j,2}-p_{j,1}} \\
%&\le d^{-(p_{j,2}-p_{j,1})}|u_j(x)-u_j(y)|^{p_{j,2}-p_{j,1}} \\
%&\le d^{-(p_{j,2}-p_{j,1})}(2\|u\|_{L^\infty(\Omega')}+1)^{\omega_p(r)}\\
%& \le c.
%\end{split}\]
%Using this estimate and
Applying Corollary~\ref{corlog} to $a=\theta(r_j)/2$, $b=\mathrm{exp}(k)$, $d=d_j$ and $\rho=2r_{j+1}$ and $r=r_j$ and setting $\tilde p_{j+1} :=\sup_{x,y\in B_{4r_{j+1}}}p(x,y)$, we have  
\[\begin{aligned}
&\fint_{2B_{j+1}}|v-(v)_{2B_{j+1}}|^{\tilde p_{j+1}} \,dx  \\
& \le  c   r_{j+1}^{-\omega_s(4r_{j+1})p^--\omega_p(4r_{j+1})}+c\max\{d,d^{-1}\}^{\omega_p(r_j)} (\sigma r_{j})^{-\omega_p(r_{j})s^--\omega_s(r_{j})p^-} \\
&\qquad+ c r_{j+1}^{-\omega_p(4r_{j+1})s^--\omega_s(4r_{j+1})p^-}\max\{d,d^{-1}\}^{\omega_p(r_j)}  \frac{r_{j+1}^{s_{0}p_{0}}}{d_j^{p_0-1}}T(|u_j|+\|u_j\|_{L^\infty(B_j)},r_j,4r_{j+1}) \\
&\overset{\eqref{sigmap2p1},\eqref{logholderj},\eqref{dp2p1}}{\le} c +c \frac{r_{j+1}^{s_{0}p_{0}}}{d_j^{p_0-1}}T(|u_j|+\|u_j\|_{L^\infty(B_j)},r_j,r_j)  \overset{\eqref{tailj}}{\le} c.
\end{aligned}\] 
In particular, we have 
\begin{equation}\label{L1estimate}
\fint_{2B_{j+1}}|v-(v)_{2B_{j+1}}| \,dx  \le c.
\end{equation}

By the definition of $v$ in \eqref{vdef} and \eqref{density}, we have 
\[
k = \frac{1}{|2B_{j+1}\cap\{u_j\ge \theta(r_j)/2\}|}\int_{2B_{j+1}\cap \{v=0\}} [k-v]\,dx  \le 2\fint_{2B_{j+1}} [k-v]\,dx = 2[k-(v)_{2B_{j+1}}].
\]
This and \eqref{L1estimate} imply 
\[\begin{aligned}
\frac{|2B_{j+1}\cap \{v=k\}|}{|2B_{j+1}|} \le \frac{2}{k|2B_{j+1}|} \int_{2B_{j+1}\cap\{v=k\}}[k-(v)_{2B_{j+1}}]\,dx \le \frac{2}{k} \fint_{2B_{j+1}}|v-(v)_{2B_{j+1}}|\,dx \le \frac{c}{k}.
\end{aligned}\]

Here, we choose
\begin{equation}\label{kdef}
k=\ln\left(\frac{\theta(r_j)/2+\epsilon \theta(r_j)}{3\epsilon\theta(r_j) }\right) = \ln\left(\frac{\theta(r_j)/2+d_j}{2\epsilon\theta(r_j)+d_j }\right),
\end{equation}
and assume that $\sigma\in(0,\frac14)$ satisfies 
\begin{equation}\label{sigma1}
\sigma \le 6^{-\frac{4(p^{+}-1)}{s^-p^{-}}} \quad \left(\ \ \Longleftrightarrow\ \ \ln 6 \le \frac{s^{-}p^{-}}{4(p^{+}-1)}\ln\bigg(\frac{1}{\sigma}\bigg)\ \ \right).
\end{equation}
Then we see that $\{v=k\}=\{u_j\le  2\epsilon\theta(r_j)\}=\{u_j\le  2d\}$ and
\[
k \ge \ln\bigg(\frac{1}{6\epsilon }\bigg) \overset{\eqref{alphaassumption},\eqref{depsilon}}{\ge}  \frac{s^{-}p^{-}}{2(p^{+}-1)} \ln\bigg(\frac{1}{\sigma}\bigg)-\ln 6 
\ge  \frac{s^{-}p^{-}}{4(p^{+}-1)}\ln\bigg(\frac{1}{\sigma}\bigg).
\]
Therefore, combining the above results we have
\begin{equation}\label{densitytildeB}
\frac{|2B_{j+1}\cap \{u_j\le 2d_j\}|}{|2B_{j+1}|} =\frac{|2B_{j+1}\cap \{v=k\}|}{|2B_{j+1}|} \le \frac{c}{k} \le \frac{c_0}{\ln\left(\frac{1}{\sigma}\right)}
\end{equation}
for some $c_0>0$ depending only on $n$, $\Lambda$, $s^\pm$, $p^\pm$, and $c_{LH}$, but independent of $\sigma$. 

\textit{Step 4 (Proof of \eqref{inductionj+1}).} Finally, we  complete the proof by proving \eqref{inductionj+1}.
Set 
\[
\tilde B= 2B_{j+1}, \quad \tilde s_1 = s_{j+1,1},\quad \tilde s_2 = s_{j+1,2}, \quad 
\tilde p_1 = p_{j+1,1}, \quad \tilde p_2=p_{j+1,2}, \quad d=d_j. 
\]

For $i=0,1,2,\dots$, set
\[
\rho_i:=(1+2^{-i}) r_{j+1} , \qquad  B^i:=B_{\rho_i},
\]
\[
k_i:=(1+2^{-i}) d,
\qquad 
 w_i:=(k_i-u_j)_+,
\]
\[
A_i:=\frac{|B^i\cap \{u_j< k_i\}|}{|B^i|}=\frac{|B^i\cap \{w_i>0\}|}{|B^i|}.
\]
Notice that for every $i=0,1,2,\dots$,
\[
r_{j+1}< \rho_{i+1}\le \rho_i\le 2r_{j+1},\quad d \le k_{i+1}  \le k_i \le 2d 
\quad\text{and}\quad 
0\le w_i\le k_i \le 2d.
\]

We will prove that $A_i\to 0$ as $i\to \infty$, by choosing  $\sigma$ sufficiently small.  
Set 
\begin{equation}\label{tdef}
t:=\frac{s^-}{2}. 
\end{equation}
Note that 
\[
n-t\tilde p_1 \ge n-\frac{s_1}{2}p_2  \ge  n-\frac{s_1}{2}p_1- \frac{s^+}{2}\omega_p(r) \overset{\eqref{Rselect4}}{\ge} n-\frac{s_1p_1}{2}- \frac{s_1p_1}{4} \overset{\eqref{holderbasicass}}{>}0.
\]
By Lemma~\ref{lemSPineq} and Lemma~\ref{lempxpq}
\begin{equation}\label{inductionpf1}\begin{aligned}
&A_{i+1}^{\frac{(n-t \tilde p_{1})}{n}} (k_i-k_{i+1})^{\tilde p_{1}} = \left(\frac{1}{|B^{i+1}|}
\int_{B^{i+1}\cap\{u_j<k_{i+1}\}}(k_i-k_{i+1})^{\frac{n\tilde p_{1}}{(n-t  \tilde p_{1})}}\,dx \right)^{\frac{(n-t  \tilde p_{1})}{n}}\\
 & \le c \left(
\fint_{B^{i+1}}w_i^{\frac{n\tilde p_{1}}{(n-t \tilde p_{1})}}\,dx \right)^{\frac{(n-t \tilde p_{1})}{n}}\\
& \le c \frac{\rho_{i+1}^{t \tilde p_{1}}}{|B^{i+1}|}
\varrho_{t,\tilde p_1}(w_i, B^{i+1}) + c\fint_{B^{i+1}} w_i^{\tilde p_{1}}\,dx\\
&\le c\frac{r_{j+1}^{\tilde s_1 \tilde p_{1}}}{|B^{i+1}|} \varrho_{s(\cdot,\cdot), p(\cdot,\cdot)}(w_i, B^{i+1}) +c  \left(r_{j+1}^{\tilde s_1 \tilde p_{1}} + d^{\tilde p_1} \right)A_i \\
&\le c\frac{r_{j+1}^{\tilde s_1 \tilde p_{1}} }{|B^{i+1}|}\varrho_{s(\cdot,\cdot), p(\cdot,\cdot)}(w_i, B^{i+1}) +c d^{p_0} A_i .
\end{aligned}\end{equation}
In the last inequality we used \eqref{dp2p1} and the following estimate:
\[
r_{j+1}^{\tilde s_{1}} = \sigma^{\tilde s_{1}(j+1)}r^{\tilde s_{1}} \le   \sigma^{s^-j}\sigma^{s^-} r^{s^-} \overset{\eqref{Rselect6},\eqref{alphaassumption}}{\le} c\sigma^{\alpha j} \sigma^{\frac{s^+p^+}{p^+-1}} K_0 \overset{\eqref{depsilon},\eqref{thetaj}}{\le} cd.
\]
We then estimate the double integral on the right hand side. Applying the first inequality in \eqref{caccio} to $u=u_j$, $k=k_i$, $\rho=\rho_{i+1}$ and $r=\rho_{i}$,
\[\begin{split}
&\varrho_{s(\cdot,\cdot),p(\cdot,\cdot)}(w_i, B^{i+1})\\ &\le \frac{c}{(\rho_{i}-\rho_{i+1})^{\tilde p_{2}}} \int_{B^{i}} \int_{B^{i}}  \frac{\max\{w_i(x)^{p(x,y)},w_i(y)^{p(x,y)}\}}{|x-y|^{n+s(x,y)p(x,y)-\tilde p_{1}}}\,dy\, dx\\
&\qquad + c \frac{\rho_{i}^{n+\tilde s_2 \tilde p_{2}}}{(\rho_{i}-\rho_{i+1})^{n+\tilde s_2 \tilde p_{2}}} \left(  \int_{B^i}w_i \,dx \right) T(w_i,\rho_{i},\rho_{i})\\
&\overset{\eqref{dp2p1}}{\le} c 2^{ip^{+}}r_{j+1}^{-\tilde p_{2}} d^{p_0} \int_{B^i\cap \{w_i>0\}} \int_{B_{4r_j}(x)}  \frac{1}{|x-y|^{n+\tilde s_2 \tilde p_{2}-\tilde p_{1}}}\,dy\, dx\\
&\qquad + c 2^{i(n+sp^{+})}d|B^i\cap \{w_i>0\}| T(w_i,\rho_i,\rho_i)\\
&\overset{\eqref{sigmap2p1}\eqref{logholderj}}{\le} c 2^{ip^{+}}r_{j+1}^{-\tilde s_1 \tilde p_1}   d^{p_{0}} |B^i\cap \{w_i>0\}| + c 2^{i(n+s^+p^{+})} d|B^i\cap \{w_i>0\}| T(w_i,\rho_i,\rho_i).
\end{split}\]
In the last inequality, we also use the fact that $\tilde s_2 \tilde p_2 -\tilde p_1<0$ by \eqref{Rselect4}.  
Moreover, for every $ x\in B^i\subset B_j$,  
\[\begin{aligned}
& \int_{\R^n\setminus B^i} \frac{|w_i(y)|^{p(x,y)-1}}{|y-x_0|^{n+s(x,y)p(x,y)}}\,dy \\
& \le \int_{B_{j}\setminus B_{j+1}} \frac{|w_i(y)|^{p(x,y)-1}}{|y-x_0|^{n+s(x,y)p(x,y)}}\,dy + \int_{\R^n\setminus B_j} \frac{|w_i(y)|^{p(x,y)-1}}{|y-x_0|^{n+s(x,y)p(x,y)}}\,dy \\
 & \le c \bigg(\sup_{x,y\in B^j}d^{p(x,y)-1}\bigg) \int_{B_{j}\setminus B_{j+1}} \frac{1}{|y-x_0|^{n+s_{j,2} p_{j,2}}}\,dy   + T(\theta(r_j),r_j,r_j)  \\
& \overset{\eqref{dp2p1},\eqref{jtail}}{\le}  c d^{p_0-1}  r_{j+1}^{-s_{j,2}p_{j,2}} + cr_j^{-s_0p_{0}} \sigma^{-\alpha(p_0-1)} \theta(r_j)^{p_{0}-1}  \\
& \overset{\eqref{depsilon}}{\le}   cd^{p_0-1} (r_{j+1}^{-s_{j,2}p_{j,2}}+r_{j+1}^{-s_0 p_{0}}) \overset{\eqref{sigmap2p1}\eqref{logholderj}}{\le} cd^{p_0-1} r_{j+1}^{-\tilde s_1\tilde p_1} .
\end{aligned}
\]
Therefore we have
\[
\varrho_{s(\cdot,\cdot),p(\cdot,\cdot)}(w_i, B^{i+1}) \le  c 2^{i(n+p^{+})} d^{p_0}|B^i\cap \{w_i>0\}| r_{j+1}^{-\tilde s_{1}\tilde p_{1}}.
\]
Inserting this into \eqref{inductionpf1}, we obtain
\[
A_{i+1}^{\frac{(n-t \tilde p_{1})}{n}} (k_i-k_{i+1})^{\tilde p_{1}} \le c 2^{i(n+p^{+})} d^{p_0} A_i.
\]
Moreover, by \eqref{dp2p1} and \eqref{tdef},
\[
A_{i+1}^{1-\frac{s^-p^-}{2n}} 2^{-(i+1)p^{+}}d^{p_{0}}
 \le c A_{i+1}^{\frac{(n-t \tilde p_{1})}{n}} [2^{-i-1}d]^{\tilde p_{1}} = c  A_{i+1}^{\frac{(n-t \tilde p_{1})}{n}} (k_i-k_{i+1})^{\tilde p_{1}}.
\]
The preceding two estimates yield
\[
A_{i+1} \le c_1 2^{i\frac{2n(n+2p^{+})}{2n-s^-p^-}} A_i^{1+\frac{s^-p^-}{2n-s^-p^-}}
\]
for some $c_1>0$ depending only on $n$, $s^\pm$, $p^\pm$, $\Lambda$ and $c_{LH}$, and  independent of $\sigma$. Therefore, by Lemma~\ref{lemseq}
if 
\begin{equation}\label{A0ineq}
A_0=\frac{| 2B_{j+1}\cap \{u_j\le 2d\}|}{|2B_{j+1}|} \le c_1^{-\frac{2n-s^-p^-}{s^-p^-}} 2^{-\frac{2n(n+2p^{+})(2n-s^-p^-)}{(s^-p^-)^2}},
\end{equation}
then $A_i\to 0$ as $i\to \infty$ which implies 
\begin{equation}\label{ujBj+1}
|B_{j+1}\cap \{u_j < d\}| =0  
\quad \Longleftrightarrow\quad 
u_j\ge d= \epsilon \theta(j)\ \ \text{in } B_{j+1}.
\end{equation}
Here, we assume that $\sigma <\frac{1}{4}$ satisfies
\begin{equation}\label{sigma2}
\sigma \le \exp\left(-c_0c_1^{\frac{2n-s^-p^-}{s^-p^-}} 2^{\frac{2n(n+2p^{+})(2n-s^-p^-)}{(s^-p^-)^2}}\right),
\end{equation}
where $c_0$ is in \eqref{densitytildeB}. Then \eqref{densitytildeB}  implies  \eqref{A0ineq} hence we have \eqref{ujBj+1}. 

Consequently, we prove
\[
\theta(r_{j+1}) = \sup_{B_{j+1}} u_j -\inf_{B_{j+1}} u_j  \overset{\eqref{supuj},\eqref{ujBj+1}}{\le} (1-\epsilon)\theta(r_j) 
 \overset{\eqref{inductioni},\eqref{depsilon}}{\le} (1-\sigma^{\frac{s_0p_0}{p^+-1}-\alpha})\sigma^{-\alpha} K_{j+1}
 \overset{\eqref{alpha3}}{\le}  K_{j+1}.
\]
\end{proof}

Finally, we prove  Theorem~\ref{mainthm}. In fact, we almost prove the theorem in the preceding lemma.

\begin{proof}[Proof of Theorem~\ref{mainthm}]
Let $\Omega'\Subset\Omega$. Fix any $B_r\Subset\Omega'$ with $r>0$ satisfying the assumptions stated  in the beginning of the section. Set $\inf_{x,y\in B_r}s(x,y)=s_1$ and $\inf_{x,y\in B_r}p(x,y)=p_1$. If  
\[
n-s_1p_1\le -\frac{s^-p^-}{4},
\] 
then one can find $t\in (0,s_1)$ such that 
\[
n-tp_1\le -\frac{s^-p^-}{8}.
\] 
Then in view of Lemma~\ref{lempxpq} and \cite[Theorem 8.2]{DPV1}, we have $u\in  W^{t,p_1}(B_r)\subset C^{0,\beta}(B_r)$ with $\beta=\frac{tp_1-n}{p_1}\ge \frac{s^-p^-}{8p^+}$ hence  $u\in C^{0,\frac{s^-p^-}{8p^+}}(B_r)$.

On the other hand, if
\[
n-s_1p_1> -\frac{s^-p^-}{4}.
\]  
The preceding Lemma yields $u\in C^{0,\alpha}(B_r)$ for some $\alpha\in (0,1)$ satisfying \eqref{alphaassumption}.  

Therefore by standard covering argument we have that  $u\in C^{0,\alpha}_{\mathrm{loc}}(\Omega')$ for some $\alpha\in(0,1)$ depending on $n$, $s^\pm$, $p^\pm$, $\Lambda$ and $c_{LH}$. Since $\Omega'\Subset\Omega$ is arbitrary, we have the conclusion.
\end{proof}

%\begin{remark}
%We emphasize that in the above theorem no regularity condition on $s(\cdot,\cdot)$ and $p(\cdot,\cdot)$ is assumed. Therefore,  as a special example, Theorem~\ref{thmexistence} can also provide the existence and the uniqueness of the minimizer of the following nonlocal double phase problem 
%\[
%\iint_{\mathcal C_\Omega} \left[\frac{|u(x)-u(y)|^p}{|x-y|^{n+sp}}+a(x,y)\frac{|u(x)-u(y)|^q}{|x-y|^{n+tq}}\right]
%\]\,dx\,dy,
%where $p,q\in(1,\infty)$, $s,t\in (0,1)$, and $0\le a(\cdot,\cdot)\le L$. 
%\end{remark}

%%%%%%%%%%%%%%%%%%%%%%%%%%%%%%%%%%%%%%%%%%%%%%%%%%%%%%%%%%%%%%%%%%%%%%%%
%%%%%%%%%%%%%%%%%%%%%%%%%%%%%%%%%%%%%%%%%%%%%%%%%%%%%%%%%%%%%%%%%%%%%%%%
%%%%%%%%%%%%%%%%%%%%%%%%%%%%%%%%%%%%%%%%%%%%%%%%%%%%%%%%%%%%%%%%%%%%%%%%

\bibliographystyle{amsplain}

\end{document}